\documentclass[onefignum,onetabnum]{siamart171218}

\usepackage{amsfonts}
\usepackage{graphicx}
\usepackage{epstopdf}
\usepackage{algorithmic}
\usepackage{amssymb}
\ifpdf
  \DeclareGraphicsExtensions{.eps,.pdf,.png,.jpg}
\else
  \DeclareGraphicsExtensions{.eps}
\fi

\makeatletter
\newcommand*{\addFileDependency}[1]{
  \typeout{(#1)}
  \@addtofilelist{#1}
  \IfFileExists{#1}{}{\typeout{No file #1.}}
}
\makeatother

\ifpdf
\hypersetup{
  pdftitle={An Example Article},
  pdfauthor={D. Doe, P. T. Frank, and J. E. Smith}
}
\fi

\usepackage{enumitem}
\usepackage{tikz}
\usepackage{ifthen}
\usetikzlibrary{shapes}
\usetikzlibrary{calc}

\newcommand{\bx}{\boldsymbol{x}}
\newcommand{\blambda}{\boldsymbol{\lambda}}
\newcommand{\bomega}{\boldsymbol{\omega}}
\newcommand{\btheta}{\boldsymbol{\theta}}
\newcommand{\bu}{\boldsymbol{u}}

\newcommand{\bV}{\boldsymbol{V}}
\newcommand{\bX}{\boldsymbol{X}}
\newcommand{\bU}{\boldsymbol{U}}

\newcommand{\bZ}{\boldsymbol{Z}}

\newcommand{\bpi}{\boldsymbol{\pi}}
\newcommand{\bphi}{\boldsymbol{\phi}}

\newcommand{\bQ}{\boldsymbol{Q}}
\newcommand{\bPhi}{\boldsymbol{\Phi}}
\newcommand{\bK}{\boldsymbol{K}}
\newcommand{\bP}{\boldsymbol{P}}

\newcommand{\bT}{\boldsymbol{T}}
\newcommand{\bF}{\boldsymbol{F}}

\newcommand{\bG}{\boldsymbol{G}}

\newcommand{\bA}{\boldsymbol{A}}
\newcommand{\bL}{\boldsymbol{L}}

\newcommand{\bB}{\boldsymbol{B}}
\newcommand{\bC}{\boldsymbol{C}}

\newcommand{\bH}{\boldsymbol{H}}
\newcommand{\bR}{\boldsymbol{R}}
\newcommand{\bn}{\boldsymbol{n}}
\newcommand{\bm}{\boldsymbol{m}}

\newcommand{\bM}{\boldsymbol{M}}

\newcommand{\bI}{\boldsymbol{I}}
\newcommand{\obx}{\overline{\bx}}

\newcommand{\obPhi}{\overline{\bPhi}}
\newcommand{\obK}{\overline{\bK}}

\newcommand{\okappa}{\overline{\kappa}}

\newcommand{\olambda}{\overline{\lambda}}
\newcommand{\ulambda}{\underline{\lambda}}

\newcommand{\bzero}{\boldsymbol{0}}
\newcommand{\st}{\mathop{\text{\normalfont s.t.}}}
\newcommand{\diag}{\mathop{\text{\normalfont diag}}}

\newcommand{\cG}{\mathcal{G}}

\newcommand{\cV}{\mathcal{V}}

\newcommand{\cE}{\mathcal{E}}

\newcommand{\cK}{\mathcal{K}}

\newcommand{\cN}{\mathcal{N}}
\newcommand{\cI}{\mathcal{I}}
\newcommand{\cJ}{\mathcal{J}}

\newsiamthm{assumption}{Assumption}
\newsiamremark{example}{Example}
\newsiamremark{remark}{Remark}

\allowbreak
\allowdisplaybreaks

\newenvironment{bmatrixfn}
{
\begin{footnotesize}
\begin{bmatrix}
}
{ 
\end{bmatrix} 
\end{footnotesize}
}

\begin{document}

\title{Near-Optimal Distributed Linear-Quadratic Regulator for Networked Systems\thanks{
    Submitted to the editors on \today.
    \funding{
      This material is based upon work supported by the U.S. Department of Energy, Office of Science, Office of Advanced Scientific Computing Research (ASCR) under Contract DE-AC02-06CH11347.}
  }
}

\author{
  Sungho Shin\thanks{Mathematics and Computer Science Division, Argonne National Laboratory, Lemont, IL
    (\email{sshin@anl.gov}, \email{anitescu@mcs.anl.gov}).}
  \and
  Yiheng Lin\thanks{California Institute of Technology, Pasadena, CA
    (\email{yihengl@caltech.edu}, \email{adamw@caltech.edu}).
  }
  \and
  Guannan Qu\thanks{Department of Electrical and Computer Engineering, Carnegie Mellon University, Pittsburgh, PA
    (\email{gqu@andrew.cmu.edu})
  }
  \and 
  {\\Adam Wierman}\footnotemark[3]
  \and
  Mihai Anitescu\footnotemark[2] \thanks{Department of Statistics, University of Chicago, Chicago, IL}
}

\headers{Near-Optimal Distributed LQR for Networked Systems}{S. Shin, Y. Lin, G. Qu, A. Wierman, and M. Anitescu}

\maketitle

\begin{abstract}
  This paper studies the trade-off between the degree of decentralization and the performance of a distributed controller in a linear-quadratic control setting. We study a system of interconnected agents over a graph and a distributed controller, called $\kappa$-distributed control, which lets the agents make control decisions based on the state information within distance $\kappa$ on the underlying graph. This controller can tune its degree of decentralization using the parameter $\kappa$ and thus allows a characterization of the relationship between decentralization and performance. We show that under mild assumptions, including stabilizability, detectability, and a {subexponentially} growing graph condition, the performance difference between $\kappa$-distributed control and  centralized optimal control becomes exponentially small in $\kappa$. This result reveals that distributed control can achieve near-optimal performance with a moderate degree of decentralization, and thus it is an effective controller architecture for large-scale networked systems. 
\end{abstract}



\section{Introduction}\label{sec:intro}

Because of the increasing complexity of networked systems, distributed control has gained substantial attention in the literature \cite{ho1972team,bamieh2002distributed,motee2008optimal,andreasson2014distributed,furieri2020sparsity,wang2019system}. Distributed control aims to design a set of local controllers that cooperatively optimize the systemwide performance while  having access only to local information within a prescribed range. This architecture has advantages over its centralized counterpart in robustness (failure in a component does not cause system failure), computation (online computation load is small), privacy (global information sharing is not needed), and implementation (less and shorter communication). Furthermore, it has advantages over its decentralized counterpart in performance (communication between agents mitigates the impact of decentralization).

However, the synthesis of an optimal distributed controller is a challenging problem. In particular, the optimal distributed policy can be nonlinear even for a simple {stochastic linear-quadratic control setting} \cite{witsenhausen1968counterexample}, and synthesizing an optimal distributed controller is intractable in general \cite{blondel2000survey}. Several works have studied making the problem tractable by imposing additional structural assumptions, for example,  nested information structure \cite{ho1972team}, finite-dimensional linear policy \cite{maartensson2009gradient,bamieh2002distributed}, quadratic invariance \cite{rotkowitz2005characterization}, and sparsity invariance \cite{furieri2020sparsity}.  Other works have studied linear matrix inequality conditions for the well-posedness, stability, and performance of distributed controllers \cite{langbort2004distributed,d2003distributed,dullerud2004distributed}.  Additionally, recent works on system-level synthesis have proposed a technique to synthesize localized control policy in a scalable manner \cite{wang2019system,anderson2019system}. Even for such optimal distributed policies, however, the nominal performance can be significantly worse than that of the centralized optimal policy because of the structural constraints imposed on the local controllers \cite{kariwala2007fundamental}.

An important gap in the literature between  centralized optimal control and distributed controller synthesis  is the limited understanding of how much performance loss is incurred by decentralization. The general relationship between decentralization and control performance has been studied in various contexts, such as decentralized control \cite{sourias1995best,kariwala2007fundamental,cui2002performance}, cooperative coverage problems \cite{meng2021price}, and multiagent optimization \cite{wang2017vanishing}, but the trade-off relationship between {\it the degree of decentralization} (determined by the structural constraint imposed on the controller) and {\it the distributed controller's performance} has not been reported in the literature. Since  myriad  possible communication structures for distributed control exist, without a guiding principle for trading off the degree of decentralization against performance, it is difficult to appropriately balance the level of decentralization and control performance. This situation motivates an important open question: {\it Can one quantify the trade-off between decentralization and performance?} 

We seek to address this question by analyzing the performance of a {\it $\kappa$-distributed linear-quadratic regulator} (LQR). This controller allows the user to tune the degree of decentralization using the user-defined parameter $\kappa$. The system under study is  a linear system composed of agents interconnected over a graph, which appears in a wide range of applications \cite{patel2019economic,andreasson2014distributed} and has been considered in many works on distributed control \cite{motee2008optimal,wang2019system}. The graph allows us to construct a limited-range communication structure based on the distance on the graph. In particular, $\kappa$-distributed LQR implements truncated state feedback, which  considers only the agents within a prescribed distance $\kappa$, while ignoring the agents beyond that distance (Figure \ref{fig:kappa-dec}). This controller becomes more decentralized for small $\kappa$ and less decentralized for large $\kappa$. Thus, $\kappa$-distributed LQR provides a ground for investigating the trade-off between decentralization and performance. {A similar decentralized overlapping control architecture was studied in \cite[Chapter 8]{siljak2011decentralized}. In this direction, the suboptimality bounds that we discuss here have been studied \cite{siljak2011decentralized}. However, an explicit relationship between the degree of decentralization and the suboptimality bound, as we do in this work, has not been established. }

\begin{figure}
  \centering
  \begin{tikzpicture}[scale=.25]\footnotesize
  \pgfdeclarelayer{fg}
  \pgfdeclarelayer{bg}
  \pgfdeclarelayer{nodelayer}
  \pgfdeclarelayer{edgelayer}
  \pgfsetlayers{bg,edgelayer,nodelayer,fg}
	\begin{pgfonlayer}{nodelayer}
    \node [circle,fill=black,scale=.4] (0) at (-10.5, 3.75) {};
    \node [circle,fill=black,scale=.4] (1) at (-9.25, 6.25) {};
    \node [circle,fill=black,scale=.4] (2) at (-9.5, 1.25) {};
    \node [circle,fill=black,scale=.4] (3) at (-6.25, 7) {};
    \node [circle,fill=black,scale=.4] (4) at (-6.25, 2) {};
    \node [circle,fill=black,scale=.4] (5) at (-4.25, 3.75) {};
    \node [circle,fill=black,scale=.4] (6) at (-4, 6) {};
    \node [circle,fill=black,scale=.4] (7) at (-5.5, 9.25) {};
    \node [circle,fill=black,scale=.4] (8) at (-1.5, 7.25) {};
    \node [circle,fill=black,scale=.4] (9) at (-1.75, 9.5) {};
    \node [circle,fill=black,scale=.4] (10) at (1, 6.75) {};
    \node [circle,fill=black,scale=.4] (11) at (-11, 7.5) {};
    \node [circle,fill=black,scale=.4] (12) at (-11, 9.75) {};
    \node [circle,fill=black,scale=.4] (13) at (-13.75, 6.75) {};
    \node [circle,fill=black,scale=.4] (14) at (-11.5, -1) {};
    \node [circle,fill=black,scale=.4] (15) at (-13.5, 0.5) {};
    \node [circle,fill=black,scale=.4] (16) at (-13.5, -3.75) {};
    \node [circle,fill=black,scale=.4] (17) at (-8.75, -1.75) {};
    \node [circle,fill=black,scale=.4] (18) at (-8.25, -4) {};
    \node [circle,fill=black,scale=.4] (19) at (-2, 0.75) {};
    \node [circle,fill=black,scale=.4] (20) at (-1.25, 3.75) {};
    \node [circle,fill=black,scale=.4] (21) at (1, 2.25) {};
    \node [circle,fill=black,scale=.4] (22) at (-1.25, -2.25) {};
    \node [circle,fill=black,scale=.4] (23) at (2, -2.5) {};
    \node [circle,fill=black,scale=.4] (24) at (3.25, 0.5) {};
    \node [circle,fill=black,scale=.4] (25) at (-3.25, -4.25) {};
    \node [circle,fill=black,scale=.4] (26) at (5.5, 1.5) {};
    \node [circle,fill=black,scale=.4] (27) at (6, 4.25) {};
  \end{pgfonlayer}
	\begin{pgfonlayer}{edgelayer}
    \draw (1.center) to (3.center);
    \draw (1.center) to (0.center);
    \draw (12.center) to (11.center);
    \draw (11.center) to (13.center);
    \draw (11.center) to (1.center);
    \draw (7.center) to (3.center);
    \draw (3.center) to (6.center);
    \draw (6.center) to (5.center);
    \draw (5.center) to (4.center);
    \draw (4.center) to (2.center);
    \draw (2.center) to (0.center);
    \draw (2.center) to (14.center);
    \draw (14.center) to (16.center);
    \draw (14.center) to (15.center);
    \draw (14.center) to (17.center);
    \draw (17.center) to (18.center);
    \draw (4.center) to (19.center);
    \draw (5.center) to (20.center);
    \draw (20.center) to (21.center);
    \draw (19.center) to (21.center);
    \draw (19.center) to (22.center);
    \draw (22.center) to (23.center);
    \draw (23.center) to (24.center);
    \draw (24.center) to (21.center);
    \draw (24.center) to (26.center);
    \draw (26.center) to (27.center);
    \draw (6.center) to (8.center);
    \draw (8.center) to (10.center);
    \draw (8.center) to (9.center);
    \draw (22.center) to (25.center);
  \end{pgfonlayer}
  \begin{pgfonlayer}{bg}

    \begin{scope}[opacity=.15, transparency group]
      \draw[red,line width=22,line cap=round] (4.center) to (2.center);
      \draw[red,line width=22,line cap=round] (2.center) to (0.center);
      \draw[red,line width=22,line cap=round] (2.center) to (14.center);
      \draw[red,line width=22,line cap=round] (5.center) to (4.center);
      \draw[red,line width=22,line cap=round] (4.center) to (2.center);
      \draw[red,line width=22,line cap=round] (4.center) to (19.center);
      \draw[red,line width=22,line cap=round] (1.center) to (0.center);
      \draw[red,line width=22,line cap=round] (2.center) to (0.center);
      \draw[red,line width=22,line cap=round] (2.center) to (14.center);
      \draw[red,line width=22,line cap=round] (14.center) to (16.center);
      \draw[red,line width=22,line cap=round] (14.center) to (15.center);
      \draw[red,line width=22,line cap=round] (14.center) to (17.center);
    \end{scope}

    \begin{scope}[opacity=.15, transparency group]
      \draw[blue,line width=22,line cap=round] (20.center) to (21.center);
      \draw[blue,line width=22,line cap=round] (19.center) to (21.center);
      \draw[blue,line width=22,line cap=round] (24.center) to (21.center);
      \draw[blue,line width=22,line cap=round] (5.center) to (20.center);
      \draw[blue,line width=22,line cap=round] (20.center) to (21.center);
      \draw[blue,line width=22,line cap=round] (4.center) to (19.center);
      \draw[blue,line width=22,line cap=round] (19.center) to (21.center);
      \draw[blue,line width=22,line cap=round] (19.center) to (22.center);
      \draw[blue,line width=22,line cap=round] (23.center) to (24.center);
      \draw[blue,line width=22,line cap=round] (24.center) to (21.center);
      \draw[blue,line width=22,line cap=round] (24.center) to (26.center);
    \end{scope}

    \begin{scope}[opacity=.3, transparency group]
      \draw[red,line width=15,line cap=round] (4.center) to (2.center);
      \draw[red,line width=15,line cap=round] (2.center) to (0.center);
      \draw[red,line width=15,line cap=round] (2.center) to (14.center);
    \end{scope}
    
    \begin{scope}[opacity=.3, transparency group]
      \draw[blue,line width=15,line cap=round] (20.center) to (21.center);
      \draw[blue,line width=15,line cap=round] (19.center) to (21.center);
      \draw[blue,line width=15,line cap=round] (24.center) to (21.center);
    \end{scope}
\end{pgfonlayer}
  \begin{pgfonlayer}{fg}
    \node [circle,draw=red,line width=2,scale=1.5,label={[red]below:$i$},label={[red!80!white,xshift=-7,yshift=13]left:$\cN_\cG^1[i]$},label={[red!40!white,xshift=-22,yshift=-18]left:$\cN_\cG^2[i]$}] (red) at (2) {};
    \node [circle,draw=blue,line width=2,scale=1.5,label={[blue]below:$j$},label={[blue!80!white,xshift=-2,yshift=10]right:$\cN_\cG^1[j]$},label={[blue!40!white,xshift=15,yshift=-25]right:$\cN_\cG^2[j]$}] (blue) at (21) {};
  \end{pgfonlayer}
\end{tikzpicture}

  \caption{Illustration of the communication structure of $\kappa$-distributed control for $\kappa=1$ and $2$.}\label{fig:kappa-dec}
\end{figure}
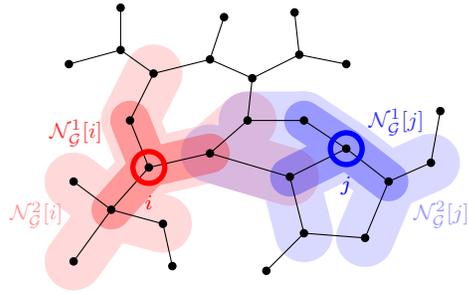

In this work we show that the nominal performance of the $\kappa$-distributed policy, measured by a  quadratic performance criterion, becomes exponentially close to that of the centralized optimal policy as $\kappa$ is increased. Thus, one can indeed improve the performance of the distributed controller by {\it enforcing a less restrictive information exchange structure}. Moreover, the exponential relationship reveals that distributed control can achieve {\it near-optimal performance} with a moderate degree of decentralization. This result manifests the effectiveness of the distributed control architecture for the control of large-scale networked systems. In addition,  our result is obtained without making restrictive structural assumptions; rather, it  relies only on the standard assumptions in linear system theory (stabilizability and detectability) and a mild assumption on the graph topology ({subexponential} growth condition). These assumptions are substantially weaker than the structural assumptions employed by prior works on near-optimal distributed control, such as decentrally optimizability \cite{yasuda1980decentrally} and linear periodic controller \cite{miller2013near}. Thus, our result can be applied to a wide range of networked system control problems. 

\subsection{Contributions}

The main contribution of this work is  the quantitative characterization of the performance of $\kappa$-distributed LQR. We show that the optimality gap of $\kappa$-distributed LQR, compared with  centralized optimal control, is exponentially small in $\kappa$. Specifically, under several mild assumptions including stabilizability, detectability, and {subexponentially} growing graph assumptions, we show that the $\kappa$-distributed LQR is stabilizing for sufficiently large $\kappa$ and that the optimality gap decays exponentially in $\kappa$ (Theorems \ref{thm:stab}, \ref{thm:regret}). That is, one can achieve near-optimal performance in the sense that the optimality gap can be made arbitrarily small by choosing a sufficiently large $\kappa$. This result rigorously quantifies the trade-off between performance and decentralization and, in turn, provides a guiding principle for balancing the degree of decentralization and control performance. 

The key technical result underlying our analysis is the exponential decay property of the optimal gain (Theorem \ref{thm:decay}). By leveraging the state-of-the-art perturbation bounds in graph-structured optimization \cite{shin2022exponential,motee2008optimal,qu2022scalable}, we show that the optimal gain from one agent's state to another agent's control decays exponentially with respect to the distance between the two agents. This result provides an intuitive justification for the near-optimality: As the gains from the agents more than $\kappa$ hops away are exponentially small in $\kappa$, truncating them does not cause  significant performance loss. This approach is novel in that, to the best of our knowledge, the exponential decay property has not been used for the performance analysis of distributed control. 

We complement the stability and performance results by establishing sufficient conditions for the uniformity of the main assumptions (stabilizability and detectability) in the system size (Theorem \ref{thm:uniform}). These conditions in turn guarantee the uniformity of the stability and performance results in the system size. The established sufficient conditions allow us to apply our main results to an even broader class of systems whose satisfaction of uniform stabilizability and detectability conditions might not be immediately clear. These situations may arise when only a subset of nodes are controlled or only a subset of states are observed by the performance index. The proposed sufficient conditions can serve as a design principle for large-scale systems that facilitates the application of distributed control.

\subsection{Related Works}
Our main results are built on the exponential decay property in the perturbation analysis of graph-structured optimization problems. This property was first studied in  early works on block-banded matrices \cite{demko1984decay}, wherein the authors showed the exponential decay of the entries of the inverse of band matrices. Recent works \cite{shin2020decentralized,shin2022exponential} have extended these results to analyze the sensitivity of the solutions of graph-structured optimization problems. In the context of optimal control, the exponential decay property on the time domain has been studied in \cite{lin2021perturbation,na2020exponential,xu2018exponentially,grune2020exponential}. The connection between exponential decay in graph-structured optimization and the exponential decay in optimal control problems has been made in \cite{shin2021controllability}. However, these results do not establish the exponential decay of optimal gain of LQR with respect to the spatial distance (i.e., distance on the graph).

The derivation of exponential decay in LQR gain (Theorem \ref{thm:decay}) from the results in \cite{shin2022exponential} is nontrivial;  several novel ideas are required. First, we consider a finite-horizon LQR problem for a networked system as a graph-structured optimization problem induced by a space-time graph (Figure \ref{fig:space-time-graph}, bottom). Second, we replace the terminal cost with an infinite-horizon cost-to-go function to make the solution equal to the infinite-horizon counterpart. This allows applying the result in \cite{shin2022exponential} to analyze the infinite-horizon LQR solution. Third, we establish the relationship between the uniform regularity conditions \cite[Assumption 5.2-5.3]{shin2022exponential} and system-theoretical properties to express the decay constants in terms of constants related to system-theoretical properties.

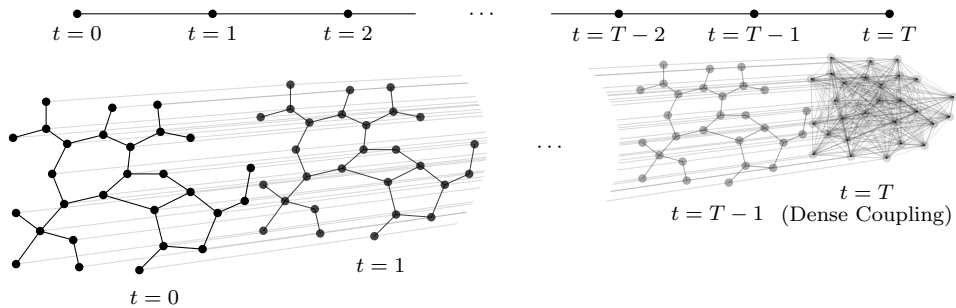
\begin{figure}[t]
  \centering
  \begin{tikzpicture}[scale=.9]\footnotesize
  \foreach \x/\l in {0/0,1/1,2/2,4/T-2,5/T-1,6/T}
  \node[circle,fill=black,scale=.4,label=below:{$t=\l$}] at (\x*2,0) {};
  \draw (0,0)--(5,0);
  \draw (7,0)--(12,0);
  \node at (6,0) {$\cdots$};
\end{tikzpicture}
  \begin{tikzpicture}[scale=.16]\footnotesize
  \foreach \z/\s/\l/\ll in {0/1/A/0,  1.2/.9/B/1, 3/.7/C/T-1, 3.8/.6/D/T}{
    \node [opacity=1-\z/4.5,circle,fill=black,scale=.4] (\l0) at (\s*-10.5+\z*16, \s*3.75+\z*2) {};
    \node [opacity=1-\z/4.5,circle,fill=black,scale=.4] (\l1) at (\s*-9.25+\z*16, \s*6.25+\z*2) {};
    \node [opacity=1-\z/4.5,circle,fill=black,scale=.4] (\l2) at (\s*-9.5+\z*16, \s*1.25+\z*2) {};
    \node [opacity=1-\z/4.5,circle,fill=black,scale=.4] (\l3) at (\s*-6.25+\z*16, \s*7+\z*2) {};
    \node [opacity=1-\z/4.5,circle,fill=black,scale=.4] (\l4) at (\s*-6.25+\z*16, \s*2+\z*2) {};
    \node [opacity=1-\z/4.5,circle,fill=black,scale=.4] (\l5) at (\s*-4.25+\z*16, \s*3.75+\z*2) {};
    \node [opacity=1-\z/4.5,circle,fill=black,scale=.4] (\l6) at (\s*-4+\z*16, \s*6+\z*2) {};
    \node [opacity=1-\z/4.5,circle,fill=black,scale=.4] (\l7) at (\s*-5.5+\z*16, \s*9.25+\z*2) {};
    \node [opacity=1-\z/4.5,circle,fill=black,scale=.4] (\l8) at (\s*-1.5+\z*16, \s*7.25+\z*2) {};
    \node [opacity=1-\z/4.5,circle,fill=black,scale=.4] (\l9) at (\s*-1.75+\z*16, \s*9.5+\z*2) {};
    \node [opacity=1-\z/4.5,circle,fill=black,scale=.4] (\l10) at (\s*1+\z*16, \s*6.75+\z*2) {};
    \node [opacity=1-\z/4.5,circle,fill=black,scale=.4] (\l11) at (\s*-11+\z*16, \s*7.5+\z*2) {};
    \node [opacity=1-\z/4.5,circle,fill=black,scale=.4] (\l12) at (\s*-11+\z*16, \s*9.75+\z*2) {};
    \node [opacity=1-\z/4.5,circle,fill=black,scale=.4] (\l13) at (\s*-13.75+\z*16, \s*6.75+\z*2) {};
    \node [opacity=1-\z/4.5,circle,fill=black,scale=.4] (\l14) at (\s*-11.5+\z*16, \s*-1+\z*2) {};
    \node [opacity=1-\z/4.5,circle,fill=black,scale=.4] (\l15) at (\s*-13.5+\z*16, \s*0.5+\z*2) {};
    \node [opacity=1-\z/4.5,circle,fill=black,scale=.4] (\l16) at (\s*-13.5+\z*16, \s*-3.75+\z*2) {};
    \node [opacity=1-\z/4.5,circle,fill=black,scale=.4] (\l17) at (\s*-8.75+\z*16, \s*-1.75+\z*2) {};
    \node [opacity=1-\z/4.5,circle,fill=black,scale=.4] (\l18) at (\s*-8.25+\z*16, \s*-4+\z*2) {};
    \node [opacity=1-\z/4.5,circle,fill=black,scale=.4] (\l19) at (\s*-2+\z*16, \s*0.75+\z*2) {};
    \node [opacity=1-\z/4.5,circle,fill=black,scale=.4] (\l20) at (\s*-1.25+\z*16, \s*3.75+\z*2) {};
    \node [opacity=1-\z/4.5,circle,fill=black,scale=.4] (\l21) at (\s*1+\z*16, \s*2.25+\z*2) {};
    \node [opacity=1-\z/4.5,circle,fill=black,scale=.4] (\l22) at (\s*-1.25+\z*16, \s*-2.25+\z*2) {};
    \node [opacity=1-\z/4.5,circle,fill=black,scale=.4] (\l23) at (\s*2+\z*16, \s*-2.5+\z*2) {};
    \node [opacity=1-\z/4.5,circle,fill=black,scale=.4] (\l24) at (\s*3.25+\z*16, \s*0.5+\z*2) {};
    \node [opacity=1-\z/4.5,circle,fill=black,scale=.4] (\l25) at (\s*-3.25+\z*16, \s*-4.25+\z*2) {};
    \node [opacity=1-\z/4.5,circle,fill=black,scale=.4] (\l26) at (\s*5.5+\z*16, \s*1.5+\z*2) {};
    \node [opacity=1-\z/4.5,circle,fill=black,scale=.4] (\l27) at (\s*6+\z*16, \s*4.25+\z*2) {};
    \draw[opacity=1-\z/4.5] (\l1.center) to (\l3.center);
    \draw[opacity=1-\z/4.5] (\l1.center) to (\l0.center);
    \draw[opacity=1-\z/4.5] (\l12.center) to (\l11.center);
    \draw[opacity=1-\z/4.5] (\l11.center) to (\l13.center);
    \draw[opacity=1-\z/4.5] (\l11.center) to (\l1.center);
    \draw[opacity=1-\z/4.5] (\l7.center) to (\l3.center);
    \draw[opacity=1-\z/4.5] (\l3.center) to (\l6.center);
    \draw[opacity=1-\z/4.5] (\l6.center) to (\l5.center);
    \draw[opacity=1-\z/4.5] (\l5.center) to (\l4.center);
    \draw[opacity=1-\z/4.5] (\l4.center) to (\l2.center);
    \draw[opacity=1-\z/4.5] (\l2.center) to (\l0.center);
    \draw[opacity=1-\z/4.5] (\l2.center) to (\l14.center);
    \draw[opacity=1-\z/4.5] (\l14.center) to (\l16.center);
    \draw[opacity=1-\z/4.5] (\l14.center) to (\l15.center);
    \draw[opacity=1-\z/4.5] (\l14.center) to (\l17.center);
    \draw[opacity=1-\z/4.5] (\l17.center) to (\l18.center);
    \draw[opacity=1-\z/4.5] (\l4.center) to (\l19.center);
    \draw[opacity=1-\z/4.5] (\l5.center) to (\l20.center);
    \draw[opacity=1-\z/4.5] (\l20.center) to (\l21.center);
    \draw[opacity=1-\z/4.5] (\l19.center) to (\l21.center);
    \draw[opacity=1-\z/4.5] (\l19.center) to (\l22.center);
    \draw[opacity=1-\z/4.5] (\l22.center) to (\l23.center);
    \draw[opacity=1-\z/4.5] (\l23.center) to (\l24.center);
    \draw[opacity=1-\z/4.5] (\l24.center) to (\l21.center);
    \draw[opacity=1-\z/4.5] (\l24.center) to (\l26.center);
    \draw[opacity=1-\z/4.5] (\l26.center) to (\l27.center);
    \draw[opacity=1-\z/4.5] (\l6.center) to (\l8.center);
    \draw[opacity=1-\z/4.5] (\l8.center) to (\l10.center);
    \draw[opacity=1-\z/4.5] (\l8.center) to (\l9.center);
    \draw[opacity=1-\z/4.5] (\l22.center) to (\l25.center);

    \node at (-6+4*\s+16*\z,-3*\s-3.5+\z*2) {$t=\ll$};
  }

  \def\z{3.8}
  \def\s{0.6}
  \node at (-6+4*\s+16*\z,-3*\s-3.5+\z*2-2) {(Dense Coupling)};
  
  \foreach \s in {0,...,27}{
    \foreach \a/\b in {A/B, B/C, C/D}{
      \draw[opacity=.15] (\a\s.center) to (\b\s.center);
    }
  }

  \foreach \s in {0,...,27}{
    \foreach \t in {0,...,27}{
      \draw[opacity=.05] (D\s.center) to (D\t.center);
    }
  }

  \fill[white] (22.5,-1) to[bend right=50] (22.5,13) to  (33.5,13) to[bend left=50] (33.5,-1) to (22.5,-1);
  \node at (31,6) {$\cdots$};
\end{tikzpicture}

  \caption{Illustration of the temporal graph (top) and space-time graph (bottom).}\label{fig:space-time-graph}
\end{figure}

The exponential decay in the optimal LQR gain has been previously studied in the work on {\it spatially decaying systems} \cite{motee2008optimal}. The authors aimed to show that for the continuous-time linear-quadratic optimal control setting and under several technical assumptions, there is exponential decay in the optimal gain. Although  later the original paper was found to have an error \cite{curtain2009comments}, it was subsequently reported that some weaker forms of the main theorem still hold true \cite{curtain2011riccati,moteeauthors}. These results  differ from ours in several ways. (i) Their results are asymptotic (the graph is infinite), whereas ours are nonasymptotic (the graph is finite); the nonasymptotic and explicit nature of our results allow the subsequent stability and performance analysis. (ii) Their setting is continuous time, whereas ours is discrete. (iii) Their proof is based on infinite-dimensional operator theory, whereas ours is based purely  on linear algebra in finite-dimensional spaces.
  
Beyond the LQR setting, the exponential decay property has been studied in various other contexts, including  combinatorial optimization \cite{gamarnik2014correlation}, multiagent reinforcement learning \cite{qu2022scalable,qu2020scalable_ave,lin2021multi}, and statistical physics \cite{bandyopadhyay2008counting,li2013correlation}, and is known to lead to effective distributed algorithm designs based on the truncation of long-range interactions \cite{gamarnik2014correlation,qu2022scalable}. While sharing some similar flavors, the analysis technique in our work deals with continuous variables, which is  different from the techniques in \cite{gamarnik2014correlation,qu2022scalable,qu2020scalable_ave,lin2021multi,bandyopadhyay2008counting,li2013correlation} that deal with discrete variables. {Moreover, the exponential decay in discrete variable settings (in particular, the Markov decision process) derives from discount factors \cite{qu2022scalable} or bounded total variation \cite{qu2020scalable_ave}, whereas the exponential decay in this paper derives from the regularity of the problem.}

\subsection{Notation}
The set of real numbers and the set of integers are denoted by $\mathbb{R}$ and $\mathbb{I}$, respectively. We define $\mathbb{I}_{A}:=\mathbb{I}\cap A$, $\mathbb{I}_{>0}:=\mathbb{I}_{(0,\infty)}$, and $\mathbb{I}_{\geq 0}:=\mathbb{I}_{[0,\infty)}$. We use the syntax $[M_1;\cdots;M_n]:=[M_1^\top\,\cdots\,M_n^\top]^\top$; $\{M_i\}_{i\in \cI}:=[M_{i_1}; \cdots; M_{i_{m}}]$; $\{M_{ij}\}_{i\in \cI,j\in \cJ}:=\{\{M_{ij}^\top\}_{j\in \cJ}^\top\}_{i\in \cI}$ for strictly ordered sets $\cI=\{i_1<\cdots<i_{m}\}$ and $\cJ=\{j_1<\cdots<j_{n}\}$. Furthermore, $v[i]$ is the $i$-th component of vector $v$; $M[i,j]$ is the $(i,j)$-th component of matrix $M$; $v[\cI]:=\{v[i]\}_{i\in \cI}$; and $M[\cI,\cJ]:=\{M[i,j]\}_{i\in \cI,j\in \cJ}$. {Vector 2-norms and induced 2-norms of matrices are denoted by $\Vert\cdot\Vert$}. For matrices $A$ and $B$, $A\succ(\succeq) B$ indicates that $A-B$ is positive (semi)-definite. The largest and smallest eigenvalues of symmetric matrices are denoted by $\olambda(\cdot),\ulambda(\cdot)$. The identity matrix is denoted by $\bI$, and the zero matrix or vector is denoted by $\bzero$. A graph $\cG=(\cV,\cE)$ is a pair of a node set $\cV$ and an edge set $\cE\subseteq \{\{i,j\}\subseteq\cV :i\neq j\}$. The closed neighborhood of $i\in\cV$ on graph $\cG$ is denoted by $\cN_\cG[i]:=\{i\}\cup\{j\in\cV:\{i,j\}\in\cE\}$. {The distance between $i$ and $j$ on graph $\cG=(\cV,\cE)$, denoted by $d_\cG(i,j)$, is the number of edges in the shortest path connecting $i$ and $j$.}

\section{Problem Setting}\label{sec:settings}

We study a networked system that can be expressed by a discrete-time, linear, time-invariant system over a graph:
\begin{subequations}\label{eqn:settings}
  \begin{align}\label{eqn:sys}
    x_{i}(t+1) = \sum_{j\in\cN_\cG[i] }\left(A_{ij} x_j(t) + B_{ij}u_j(t)\right),\quad \forall i\in \cV,\; t\in\mathbb{I}_{\geq 0}.
  \end{align}
  {
  Here $\cG:=(\cV:=\mathbb{I}_{[1,N]},\cE)$ is a graph; $x_i(t)\in\mathbb{R}^{n_{x_i}}$ and $u_i(t)\in\mathbb{R}^{n_{u_i}}$ are the state and control at node $i\in\cV$ and time index $t\in\mathbb{I}_{\geq 0}$; $A_{ij}\in\mathbb{R}^{n_{x_i}\times n_{x_j}}$; and $B_{ij}\in\mathbb{R}^{n_{x_i}\times n_{u_j}}$.} The stage performance index at node $i$ is given by a quadratic function
  \begin{align}\label{eqn:perf}
    \ell_i(x_i,u_i) := (1/2) x_i^\top \left( \sum_{j\in \cN_\cG[i]}Q_{ij} x_j \right) +  (1/2)u^\top_i \left(\sum_{j\in\cN_\cG[i]}R_{ij} u_j\right),\quad \forall i\in\cV,
  \end{align}
\end{subequations}
{where $Q_{ij}\in\mathbb{R}^{n_{x_i}\times n_{x_j}}$; $R_{ij}\in\mathbb{R}^{n_{u_i}\times n_{u_j}}$; and $Q_{ij}=Q_{ji}^\top$ and $R_{ij}=R_{ji}^\top$ for all $i,j\in\cV$. That is, the system is composed of multiple agents, interconnected over graph $\cG$, and the inter-nodal couplings are made through the non-zero $A_{ij}$, $B_{ij}$, $Q_{ij}$, and $R_{ij}$'s.}

\subsection{Centralized LQR}
The centralized LQR problem formulation for the networked system in \eqref{eqn:sys} is  as follows:
\begin{subequations}\label{eqn:lqr}
  \begin{align}
    \min_{\{\bx(t)\}_{t=0}^\infty,\{\bu(t)\}_{t=0}^\infty}\;&\sum_{t=0}^\infty (1/2)\bx(t)^\top\bQ\bx(t)+ (1/2)\bu(t)^\top\bR\bu(t)\label{eqn:lqr-obj}\\
    \st\;&\bx(0)= \obx\label{eqn:lqr-con-1}\\
                                                            &\bx(t+1) = \bA\bx(t) + \bB\bu(t),\quad \forall t=0,1,\cdots.\label{eqn:lqr-con-2}
  \end{align}
\end{subequations}
{Here we let $\bx(t):=\{x_i(t)\}_{i\in\cV}$, $\bu(t):=\{u_i(t)\}_{i\in\cV}$, $\bA:=\{A_{ij}\}_{i,j\in\cV}$, $\bB:=\{B_{ij}\}_{i,j\in\cV}$, $\bQ:=\{Q_{ij}\}_{i,j\in\cV}$, $\bR:=\{R_{ij}\}_{i,j\in\cV}$ (for convenience, $A_{ij}:=\bzero$, $B_{ij}:=\bzero$, $Q_{ij}:=\bzero$, and $R_{ij}:=\bzero$ for $\{i,j\}\notin \mathcal{E}$).} Note that \eqref{eqn:lqr-obj} is the summation of the nodal performance index \eqref{eqn:perf} over the infinite horizon and \eqref{eqn:lqr-con-1} is the {concatenation} of the nodal system dynamics \eqref{eqn:sys} into that of the full system. It is well known that under $\bQ\succeq 0$, $\bR\succ 0$, $(\bA,\bB)$-stabilizability, and $(\bA,\bQ^{1/2})$-detectability, the optimal policy $\bpi^\star(\cdot)$ for \eqref{eqn:lqr} can be expressed by a linear state feedback law $\bpi^\star(\bx) := -\bK^\star\bx$, where $\bK^\star= \{K^\star_{ij}\}_{i,j\in\cV}$ is the optimal gain matrix \cite[Theorem 2.4-2]{lewis2012optimal}. We define the state transition mapping under the optimal gain by $\bphi^\star(\bx):=\bPhi^\star\bx$, where $\bPhi^\star:=\bA-\bB\bK^\star$.

\subsection{$\kappa$-Distributed Control}
The proposed $\kappa$-distributed control policy $\bpi^\kappa(\cdot)$ is defined as $\bpi^\kappa(\bx) := -\bK^\kappa \bx$, where the optimal gain $\bK^\star$ is replaced by the {\it $\kappa$-truncated gain} $\bK^\kappa:=\{K^\kappa_{ij}\}_{i,j\in\cV}$ with $K^\kappa_{ij}:=K^\star_{ij}$ if $d_\cG(i,j)\leq \kappa$ and $\bzero$ otherwise. That is, the $\kappa$-distributed control employs a limited-range state feedback law by  taking into account only the state information within the $\kappa$-hop neighborhood $\cN^\kappa_\cG[i]:=\{j\in\cV:d_\cG(i,j)\leq \kappa\}$ (Figure \ref{fig:kappa-dec}). The $\kappa$-distributed control reduces to a purely decentralized control if $\kappa=0$, in the sense that the local controllers only have access to the state information of itself, and reduces to a purely centralized control if $\kappa\geq \max_{i,j\in\cV}d_\cG(i,j)$, in the sense that the local controllers have access to the full state information. The state transition mapping is defined by $\bphi^\kappa(\bx):=\bPhi^\kappa\bx$, where $\bPhi^\kappa:=\bA-\bB\bK^\kappa$. {In summary, the nodal feedback law of centralized and $\kappa$-distributed LQR can be expressed as follows:
\begin{align*}
  u_i(t) = \sum_{j\in\cV} -K^\star_{ij}x_j(t)\quad\text{(Centralized)};\quad
  u_i(t) = \sum_{j\in\cN_\cG^\kappa[i]} -K^\star_{ij}x_j(t)\quad\text{($\kappa$-distributed)}.
\end{align*}
That is, $\kappa$-distributed LQR is a truncated version of centralized LQR.}

  \subsection{Examples}\label{sec:eg}
  We now discuss two motivating examples.
  We use these examples to demonstrate the theoretical developments in the subsequent sections.
  {Besides the examples below, the networked system control problems also appear in various consensus problems for robotic networks (e.g., attitude alignment, flocking, coordinated decision making) \cite{ren2005survey,andreasson2014distributed} as well as the operation of storage networks \cite{lejarza2021economic}.}

  \begin{example}[Building Temperature Control \cite{patel2019economic}]\label{eg:hvac}
  Heating, ventilation, and air conditioning (HVAC) systems in large buildings with multiple zones can be modeled as networked systems. The control goal is to maintain the temperature of each zone at the desired setpoint while dealing with disturbances. We assume that all the constant disturbances (e.g., heat generation, effect of ambient conditions) are eliminated by expressing the state/control variables as deviation variables around the desired steady state; thus, the control goal becomes driving the system to the origin. The model and nodal performance index are  as follows:
  \begin{align*}
    \dot{U}_i = T_i,\quad\dot{T}_i = - \sum_{j\in\cN_{\cG}[i]} k_{ij}(T_i-T_j)  + \eta_1 u_i,\quad
    \ell_i(U_i,T_i,u_i) = \eta_2^2 U_i^2 + \eta_3^2 T_i^2 + u_i^2 ,
  \end{align*}
  where $T_i$ is the temperature of zone $i$; $U_i$ is its integrator of $T_i$; $\ell_i(\cdot)$ is the stage performance index of zone $i$; $k_{ij}$ characterizes the degree of coupling (heat exchange) between zones $i$ and $j$; {$u_i$ is the manipulated heat generation/absorption of zone $i$}; and $\eta_1,\eta_2,\eta_3$ are the constants we will use later. The integrator is not necessary for the nominal setting, but it is used in practice to deal with the offsets caused by disturbances. The explicit Euler discretization with $x_i :=[U_i;T_i]$ and time interval $\Delta t$ yields
  $A_{ii}:=
    \begin{bmatrixfn}
      1&\Delta t\\
      0&1 - (\Delta t)\sum_{j\in\cN_{\cG}[i]}k_{ij}\\
    \end{bmatrixfn}$,
    $A_{ij}:=
    \begin{bmatrixfn}
      0&0\\
      0&(\Delta t)k_{ij}\\
    \end{bmatrixfn}$,
    $B_{ii}:=
    \begin{bmatrixfn}
      0\\
      \eta_1(\Delta t)
    \end{bmatrixfn}$,
    $Q_{ii} :=
    \begin{bmatrixfn}
      \eta_2^2\\
      &\eta_3^2
    \end{bmatrixfn}$, and $R_{ii} := 1$ {(off-diagonal blocks of $\bB$, $\bQ$, $\bR$ are zero)}. Here we aim to design a distributed PI controller of the form $u_i = \sum_{j\in \cN^\kappa_\cG[i]}-(K^U_{ij} U_j + K^T_{ij} T_j)$, where $K^U_{ij}$ and $K^T_{ij}$ are the gains. \hfill$\blacksquare$
\end{example}

\begin{example}[Frequency Control of Power Systems \cite{andreasson2014distributed}]\label{eg:freq} 
  The frequency control problem in AC power networks can be expressed as a networked system control problem. Here we consider a DC-approximated model, which assumes that the voltage magnitudes are constant, the phase angle differences between adjacent buses are small, and the network is dominantly inductive.
  The control goal is to minimize the deviation of frequencies from the common reference frequency while dealing with constant disturbances. We assume that all the constant disturbances (e.g., loads) are eliminated by expressing the state/control variables as deviation variables around the desired steady state. The model and nodal performance indexes are  as follows:
  \begin{align*}
    \dot{\theta}_i = \omega_i,\quad\dot{\omega}_i  = -\sum_{j\in\cN_{\cG}[i]} k_{ij}(\theta_i-\theta_j) + \eta_1 u_i,\quad \ell_i(\theta_i,\omega_i,u_i)=\eta^2_2 \theta^2_i + \eta_3^2 \omega_i^2 + u^2_i ,
  \end{align*}
  where $\theta_i$ is the phase angle of bus $i$ (relative to a moving reference frame), $\omega_i$ is the frequency of bus $i$, $\ell_i(\cdot)$ is the stage performance index of bus $i$, $k_{ij}$ is the line susceptance, multiplied by the voltage magnitudes on both ends and divided by inertia of the generator, $u_i$ is the power injection, and $\eta_1,\eta_2,\eta_3$ are the constants we will use later. The explicit Euler discretization with $x_i :=[\theta_i;\omega_i]$ and time interval $\Delta t$ yields
  $A_{ii}:=
    \begin{bmatrixfn}
      1&\Delta t\\
      - (\Delta t)\sum_{j\in\cN_{\cG}[i]}k_{ij}&1 \\
    \end{bmatrixfn}$,
    $A_{ij}:=
    \begin{bmatrixfn}
      0&0\\
      (\Delta t)k_{ij}&0\\
    \end{bmatrixfn}$,
    $B_{ii}:=
    \begin{bmatrixfn}
      0\\
      \eta_1(\Delta t)
    \end{bmatrixfn}$,
    $Q_{ii} :=
    \begin{bmatrixfn}
      \eta_2^2\\
      &\eta_3^2
    \end{bmatrixfn}$,
    and $R_{ii} := 1$ {(off-diagonal blocks of $\bB$, $\bQ$, $\bR$ are zero)}. Here we aim to design a distributed controller of the form $u_i = \sum_{j\in \cN^\kappa_\cG[i]}-(K^\theta_{ij} \theta_j + K^\omega_{ij} \omega_j)$, where $K^\theta_{ij}$ and $K^\omega_{ij}$ are the gains. Typically, the phase angle cannot be obtained by simply integrating the frequency but must be measured via a phasor measurement unit \cite{phadke1993synchronized}. \hfill$\blacksquare$
\end{example}

\section{Exponential Decay in Optimal Gain}\label{sec:decay}
A natural requirement for the $\kappa$-truncated gain $\bK^\kappa$ to be a good approximation of the optimal gain $\bK^\star$ is that the block entries $K^\star_{ij}$ for $i,j\in\cV$ such that $d_\mathcal{G}(i, j) > \kappa$ are sufficiently small. In this section we show that under $\bQ\succeq 0$, $\bR\succ 0$, $(\bA,\bB)$-stabilizability, and $(\bA,\bQ^{1/2})$-detectability, the optimal gain $K^\star_{ij}$ from $x_j$ to $u_i$ decays exponentially in the distance $d_{\cG}(i,j)$ between $i$ and $j$. We first define three basic concepts: stability, stabilizability, and detectability. Although these are commonly used concepts in the control literature, we restate them here to explicitly introduce the associated parameters. 
\begin{definition}\label{def:main}
  We define the following for $L>0$ and $\alpha\in[0,1)$:
  \begin{enumerate}[label=(\alph*)]
  \item\label{def:stab} $\bPhi$ is $(L,\alpha)$-stable if $ \|\bPhi^t\| \leq  L \alpha^t$ for $t\in\mathbb{I}_{\geq 0}$.
  \item\label{def:stabil} $(\bA,\bB)$ is $(L,\alpha)$-stabilizable if $\exists\obK$: $\|\obK\|\leq L$ and $\bA-\bB\obK$ is $(L,\alpha)$-stable.
  \item\label{def:detect} $(\bA,\bC)$ is $(L,\alpha)$-detectable if $(\bA^\top,\bC^\top)$ is $(L,\alpha)$-stabilizable.
  \end{enumerate}
\end{definition}
Note that Definition \ref{def:main} is not altering the standard definitions. For example, the standard definition of stability (in discrete-time setting) is $sr(\bPhi)<1$, where $sr(\cdot)$ denotes the spectral radius. By \cite[Theorem 5.6.12 and Corollary 5.6.13]{horn2012matrix}, there always exist $L>0$ and $\alpha\in(0,1)$ such that $ \|\bPhi^t\| \leq  L \alpha^t$ if and only if $\bPhi$ is stable. Similarly, the standard definition of $(\bA,\bB)$-stabilizability is the existence of $\obK$ such that $sr(\bA-\bB\obK)<1$, and there always exist $L>0$ and $\alpha\in(0,1)$ such that $\|\obK\|\leq L$ and $\|(\bA-\bB\obK)^t\|\leq L\alpha^t$. We call $\obK$ in Definition \ref{def:main}\ref{def:stabil} $(L,\alpha)$-stabilizing feedback gain for $(\bA,\bB)$. Note that stabilizability and detectability are relaxations of controllability and observability. We now formally introduce the main assumption.

\begin{assumption}\label{ass:main}
  There exist $L>1$, $\alpha\in(0,1)$, and $\gamma\in(0,1)$ such that
  \begin{enumerate}[label=(\alph*)]
  \item\label{ass:bdd} $\|\bA\|,\|\bB\|,\|\bQ\|,\|\bR\|\leq L$;
  \item\label{ass:cvx} $\bR\succeq \gamma \bI$;
  \item\label{ass:stab} $(\bA,\bB)$ is $(L,\alpha)$-stabilizable;
  \item\label{ass:dect} $\bQ\succeq \bzero$ , and $(\bA,\bQ^{1/2})$ is $(L,\alpha)$-detectable.
  \end{enumerate} 
\end{assumption}
Here we use common parameters (e.g., $\|\bA\|\leq L$, $\|\bB\|\leq L$, {$\|\overline{\bK}\|\leq L$}, $\cdots$) rather than introducing parameters for each bound (e.g., $\|\bA\|\leq L_{\bA}$, $\|\bB\|\leq L_{\bB}$, {$\|\overline{\bK}\|\leq L_{\overline{\bK}}$}, $\cdots$) to reduce the notational burden. Note that we consistently use $L$ for upper bounds, $\gamma$ for strictly positive lower bounds, and $\alpha$ for the upper bounds strictly less than $1$. {We assume $L>1$, $\gamma\in(0,1)$, and $\alpha\in(0,1)$, rather than $L,\gamma,\alpha>0$ to derive the results in a simple form.} Note that Assumption \ref{ass:main} is a standard assumption used in the classical LQR literature. We are now ready to state the exponential decay result.

\begin{theorem}\label{thm:decay}
  Under Assumption \ref{ass:main}, $\|K^\star_{ij}\|\leq \Upsilon\rho^{d_\cG(i,j)}$ for $i,j\in\cV$, where
  \begin{subequations}\label{eqn:decay-constants}
    \begin{align}
      \gamma_{\bF}:=
      & \frac{(1-\alpha)^2 }{L^2 (1+L)^2 },\quad
        \gamma_{\bG}:=
        \frac{(1-\alpha)^2 \gamma}{2L^4 (1+L)^2  },\quad
        L_{\bP}
        :=\frac{L^3 ( 1+L^2)}{1-\alpha^2},\\
      \overline{\mu}:= & (2L^2_{\bH}/\gamma_{\bG} + \gamma_{\bG}+  L_{\bH})/\gamma_{\bF} \\
      \gamma_{\bH}:=
      &\left(\frac{2}{\gamma_{\bG}} + \left(1+\frac{4L_{\bH}}{\gamma_{\bG}}+\frac{4L^2_{\bH}}{\gamma_{\bG}^2} \right)\frac{L_{\bH}(1+\overline{\mu}L_{\bH})}{\gamma_{\bF}} + \overline{\mu}\right)^{-1} \\
      L_{\bH}:=
      &\max(2L+1,L_{\bP}+1),\;
        \rho:=
        \left(\dfrac{L_{\bH}^2-\gamma_{\bH}^2}{L_{\bH}^2+\gamma_{\bH}^2}\right)^{1/2},\;
        \Upsilon:=
        L_{\bH}/\gamma_{\bH}^2 \rho.
    \end{align}
  \end{subequations}
  Furthermore, $\rho\in(0,1)$, and $\Upsilon\geq 1$.
\end{theorem}

The proof is in Appendix \ref{apx:decay}. The proof of Theorem \ref{thm:decay} relies on the exponential decay property of the inverse of graph-induced banded matrices (Appendix \ref{apx:prelim}). To apply this property, we derive an equivalent finite-horizon formulation for \eqref{eqn:lqr} by using the infinite-horizon cost-to-go function, which can be obtained by the solution of the discrete-time algebraic Riccati equation (DARE). An important property of the equivalent formulation is that it is {\it graph-structured}, in the sense that the sparsity structure of the Karush–Kuhn–Tucker (KKT) matrix is induced by the space-time graph (Figure \ref{fig:space-time-graph}, bottom). Then, we observe that the optimal gain $\bK^\star$ can be obtained as a submatrix of the inverse of the graph-induced banded KKT matrix, which we show satisfies the exponential decay property. In other words,  the block component of the inverse of the KKT matrix decays exponentially with respect to the distance on the space-time graph, and the decay rate depends on the singular values of the KKT matrix, which we show are uniformly upper and lower bounded by using Assumption \ref{ass:main}. This exponential decay property in the inverse of the KKT matrix directly leads to the exponential decay in the optimal gain and finishes the proof of Theorem \ref{thm:decay}.

{Based on Theorem \ref{thm:decay}, we now can show that the $\kappa$-truncated gain $\bK^\kappa$ is an exponentially accurate approximation of the optimal gain $\bK^\star$.} We additionally require the {subexponentially} growing graph assumption.

\begin{assumption}\label{ass:poly}
  There exists a {subexponential function} $p(\cdot)$ such that
  \begin{align*}
    \left|\left\{j\in\cV: d_\cG(i,j)=d\right\}\right|\leq p(d),\quad \forall i\in\cV.
  \end{align*}
\end{assumption}

Under Assumption \ref{ass:poly}, one can derive the following corollary of Theorem \ref{thm:decay}.

\begin{corollary}\label{cor:decay}
  Under Assumptions \ref{ass:main} and \ref{ass:poly}, $\Vert\bK^\star-\bK^\kappa\Vert \leq \Psi \delta^{\kappa}$, where
  \begin{align}\label{eqn:delta-phi}
    \delta&:=(\rho+1)/2,\quad \Psi:=\bigg(\sup_{d\in\mathbb{I}_{\geq 0}} p(d) (\rho/\delta)^d\bigg)\Upsilon \delta/(1-\delta),
  \end{align}
  and $\rho,\Upsilon$ are defined in \eqref{eqn:decay-constants}. Moreover, $\delta\in(0,1)$.
\end{corollary}
The proof is in Appendix \ref{apx:cor}. One can show that even if there are multiple nodes more than $\kappa$ hops away, their accumulated effect is still exponentially small in $\kappa$, since the exponential decay in magnitude (Theorem \ref{thm:decay}) is stronger than the {subexponential} increase in their number (Assumption \ref{ass:poly}). The supremum in \eqref{eqn:delta-phi} is bounded, since the product of a {subexponential function} and an exponentially decaying function converges.

The results in this section establish the fundamental basis for the stability and performance analysis in the next section. Intuitively, Theorem \ref{thm:decay} suggests that the far-ranging interactions are negligible, so neglecting them in the control policy does not significantly degrade the control performance. In addition, since we have the explicit bounds for the difference from the optimal policy (Corollary \ref{cor:decay}), the result can be directly applied to analyze the stability and performance in the next section. 

\begin{remark}
  Our exponential decay result is different from that in \cite{motee2008optimal,moteeauthors}, where the authors have aimed to establish various {\it spatial decay} conditions in an {\it asymptotic} sense. In the particular case of exponential decay, the authors aimed to establish that there exists $\rho\in(0,1)$ such that $\sum_{j\in\cV} \rho^{-d_{\cG}(i,j)}\|K^\star_{ij}\| < +\infty$, for all $i\in \cV$ (originally presented in \cite[Theorem 6]{motee2008optimal} and corrected in \cite[Modified Version of Theorem 6]{moteeauthors}), where $\cV$ is an infinite node set. A key limitation is that one cannot obtain the exponentially decaying upper bound in the full operator space; that is, $\|\bK^\star-\bK^\kappa\|$ cannot be bounded as in Corollary \ref{cor:decay}, which serves as a crucial intermediate result for the stability and performance analysis.  \hfill$\blacksquare$
\end{remark}
\begin{remark}\label{rmk:uniform}
  Since Theorem \ref{thm:decay} establishes exponential decay in the {\it distance} on the graph, the results can be complemented by showing the uniformity of the parameters $L,\gamma,\alpha$ in the system size $N$ (the number of nodes). Typically, additional conditions are necessary to guarantee such uniformity because otherwise these parameters may tend to infinity or zero as the size of the system grows. Thus, sufficient conditions for the uniformity of Assumption \ref{ass:main} are of interest. Addressing this issue is  important, but we postpone the discussion on this matter to Section \ref{sec:uniform} in order not to bury the main points under the technical details. \hfill$\blacksquare$
\end{remark}
\begin{remark}
  Note that for any given finite graph, one can find $p(\cdot)$ that satisfies Assumption \ref{ass:poly}. However, if we consider a {\it family of systems} (and associated graphs) and want to obtain the bounds $\Psi$ and $\delta$ that apply uniformly to each system, we need $p(\cdot)$ to be uniform as well, and obtaining such $p(\cdot)$ may be impossible for certain cases. For example, consider a family of systems obtained by subgraphs of an infinite regular tree. For these systems, the number of nodes in distance $d$ can grow exponentially with $d$; thus, we cannot find a uniform {subexponential function} $p(\cdot)$. One of the sufficient conditions for the existence of such a uniform $p(\cdot)$ is that the graphs are obtained as subgraphs of a mesh in a finite-dimensional space. For instance, if the parent graph is a {$1$-dimensional} mesh, the number of nodes in distance $d$ is always not greater than $2$. In a more general case of $D$-dimensional mesh, the number of nodes in distance $d$ is $O(d^{D-1})$; thus, one can always find $p(\cdot)$ satisfying Assumption \ref{ass:poly}. \hfill$\blacksquare$
\end{remark}
{
  \begin{remark}
    We now discuss the intuitive interpretation of Theorem \ref{thm:decay}. One possible misinterpretation is that the sparsity in \eqref{eqn:settings} causes the delay in the propagation of perturbations, and it causes exponential decay. This ignores that even if the sparsity in the system causes the delay in propagation, the initial condition of any node has effects on every agent's decision at time $0$ (assuming the graph is connected). This follows from the fact that the optimal decision from \eqref{eqn:lqr} takes into account the anticipated propagation of the perturbations from distant agents and takes proactive actions; mathematically this can be seen from the fact that the inverse of the KKT matrix, which maps changes in data to changes in the optimal control, is dense. While this proactive action can be arbitrarily large in principle, Theorem \ref{thm:decay} establishes that the magnitude of such optimal proactive action decays exponentially as long as Assumption \ref{ass:main} holds. This indicates that the exponential decay does not derive from the delayed propagation of perturbations but rather comes from the regularity of the problem (imposed by Assumption \ref{ass:main}), which naturally damps the magnitude of proactive control actions against the effects of distant agents. \hfill$\blacksquare$
  \end{remark}
  \begin{remark}
    An interesting open question is whether comparable results hold for continuous-domain (continuous-time or space) optimal control problems. The temporal exponential decay for continuous-domain problems was recently established \cite{grune2020exponential,grune2019sensitivity}, but the spatial counterpart has not been reported. Studying continuous-domain problems as limiting cases of discrete-domain problems is non-trivial because the uniform regularity does not hold when refining the discretization mesh size. We leave the analysis of continuous-domain problems as a topic of future work. \hfill$\blacksquare$
  \end{remark}
  \begin{remark}
    Since many practical control problems have nonlinearity or constraints, establishing comparable results for nonlinear, constrained setting is of interest. Generalizing our results to a nonlinear setting can be done by applying the classical perturbation analysis results for nonlinear programs and generalized equations \cite{dontchev2009implicit,robinson1980strongly,na2020exponential,shin2022exponential}, and analyzing constrained problems can be done via active set analysis \cite{xu2018exponentially,xu2019exponentially,shin2021diffusing}. However, the perturbation analysis for nonlinear, constrained problems is local in nature, and strong assumptions are necessary to obtain the perturbation bound over a neighborhood of interest. We leave the analysis of nonlinear, constrained settings as a topic of future work. \hfill$\blacksquare$
  \end{remark}
}

\section{Stability and Performance Analysis}\label{sec:main}
We now establish the exponential stability and the near optimality of $\kappa$-distributed control. Stability is a prerequisite for the performance analysis because if the system is unstable, the performance index trivially becomes $+\infty$. The following theorem establishes that under Assumptions \ref{ass:main} and \ref{ass:poly} and sufficiently large $\kappa$, the state transition matrix $\bPhi^\kappa$ induced by $\bpi^\kappa(\cdot)$ is exponentially stable.

\begin{theorem}\label{thm:stab}
  Under Assumptions \ref{ass:main} and \ref{ass:poly} and $\kappa\geq \okappa$, $\bPhi^\kappa$ is $(\Omega,\beta)$-stable, where \\[-.17in]
  \begin{subequations}\label{eqn:beta-Omgea-okappa}
    \begin{align}
      \beta&:=\left(1-(1-\rho^2)/2\Upsilon^2\right)^{1/2},\quad
             \Omega:=\left(\Upsilon^2/(1-\rho^2)\right)^{1/2},\\
      \okappa&:=\log\left(\dfrac{1-\rho^2}{2\Upsilon^2\Psi L (L\Psi+2L(1+L_{\bP}L^2/\gamma))}\right)/\log\delta ,
    \end{align}
  \end{subequations}
  and $\Upsilon$, $\rho$, $L_{\bP}$, $\delta$, and $\Psi$ are defined in \eqref{eqn:decay-constants} and \eqref{eqn:delta-phi}. Furthermore, $\beta\in(0,1)$.
\end{theorem}
The proof is presented in Appendix \ref{apx:stab}. The sketch of the proof is as follows. We first quantify the convergence rate of the centralized optimal policy $\bpi^\star(\cdot)$. This analysis reveals that the centralized controller has a uniformly bounded {\it stability margin} {(only dependent on $L,\alpha,\gamma$)}. Then, by using the fact that $\kappa$-distributed policy can become arbitrarily close to $\bpi^\star(\cdot)$ (Corollary \ref{cor:decay}), we show that by choosing a sufficiently large $\kappa$, the difference incurred by decentralization can be made sufficiently small to be tolerable by the stability margin. Therefore, $\kappa$-distributed control with sufficiently large $\kappa$ can remain stable.

Now we analyze the performance of $\bpi^\kappa(\cdot)$. To formally study the performance, we first define the {\it closed-loop} performance index for initial state $\obx$ and policy $\bpi(\cdot)$:
\begin{align*}
  J(\bpi(\cdot);\obx) := \sum^\infty_{t=0}(1/2)\bx(t;\obx)^\top\bQ\bx(t;\obx)+ (1/2)\bu(t;\obx)^\top\bR\bu(t;\obx),
\end{align*}
where $\{\bx(t;\obx),\bu(t;\obx)\}_{t\in\mathbb{I}_{\geq 0}}$ denotes the state and control trajectories of the system starting from $\obx$ and controlled by $\bpi(\cdot)$. In the next theorem, we show that the optimality gap of $\bpi^\kappa(\cdot)$ decays exponentially with respect to $\kappa$.
\begin{theorem}\label{thm:regret}
  Under Assumptions \ref{ass:main} and \ref{ass:poly} and $\kappa\geq \okappa$, we have $J(\bpi^\kappa(\cdot);\obx)- J(\bpi^\star(\cdot);\obx) \leq \Gamma \delta^\kappa\|\obx\|^2$, where
  \begin{align}
    \Gamma:=\Omega^2L\Psi\left((1+LL_{\bP})(2L_{\bP}L^2/\gamma + \Psi)+2LL_{\bP}\right)/(1-\beta^2),
  \end{align}
  and $L_{\bP},\delta,\Psi,\beta,\Omega,\okappa$ are defined in \eqref{eqn:decay-constants}, \eqref{eqn:delta-phi}, and \eqref{eqn:beta-Omgea-okappa}.
\end{theorem}
The proof is in Appendix \ref{apx:regret}. The proof is built on Corollary \ref{cor:decay} and Theorem \ref{thm:stab}. Specifically, we first show that the stagewise performance loss is $O(\|\bx(t)\|^2\delta^\kappa)$;  we then show that even if we sum up these terms over the infinite horizon, we still have $O(\|\obx\|^2\delta^\kappa)$ performance loss, due to the exponential stability. 

Theorem \ref{thm:stab} says that the performance loss incurred by decentralization decays exponentially in $\kappa$. Therefore, by increasing $\kappa$, one can make the closed-loop performance of the $\kappa$-distributed controller arbitrarily close to that of the centralized optimal controller. Thus, the $\kappa$-distributed control is near-optimal in the sense that its performance can become arbitrarily close to that of the centralized control by choosing a sufficiently large $\kappa$. However, the near-optimality comes at a trade-off in increased computation (must perform matrix-vector multiplications with larger sizes) and communication (state information must be shared with the larger neighborhood). The system designer should tune $\kappa$ to balance the control performance and computation/communication loads. {Note that it is sufficient for $\kappa$ to be $O(\log\epsilon)$ to achieve $\epsilon$-performance loss. This implies that one can significantly reduce the optimality gap by increasing $\kappa$ by a moderate factor. With this relationship, it is likely that one can achieve desirable performance with moderate $\kappa$.}

The bounds in Theorems \ref{thm:decay}, \ref{thm:stab}, and \ref{thm:regret} may tend to be  conservative compared with the actual decay rates. Nevertheless, those explicit bounds provide insights into the behavior of $\kappa$-distributed LQR. First, the bounds reveal under which conditions the decay rates $\rho,\beta,\delta$ become close to $1$ (slow). For example, it is clear from their definitions that they approach $1$ as $\alpha\rightarrow 1$, since $\alpha\rightarrow 1 \implies\gamma_{\bF}\rightarrow0 \implies\gamma_{\bH}\rightarrow 0 \implies \rho,\beta,\delta\rightarrow 1$. Thus, we can anticipate that if the stabilizability and detectability are close to being violated, the decay rates in Theorems \ref{thm:decay}, \ref{thm:stab}, and \ref{thm:regret} become slow. Second, if $L,\gamma,\alpha,p(\cdot)$ are independent of the system size $N$, the constants in Theorems \ref{thm:decay}, \ref{thm:stab}, and \ref{thm:regret} are also independent of the system size. In other words, the results in Theorems \ref{thm:decay}, \ref{thm:stab}, and \ref{thm:regret} will hold for arbitrarily large problems.

We now revisit Examples \ref{eg:hvac} and \ref{eg:freq} and discuss the validity of the assumptions.

{\it Example} \ref{eg:hvac} (revisited). The HVAC system model can be expressed by
\begin{align*}
  \begin{bmatrixfn}
    \bU(t+1)\\
    \bT(t+1)
  \end{bmatrixfn}=
  \begin{bmatrixfn}
    \bI&(\Delta t)\bI\\
    \bzero &\bI - (\Delta t)\bL\\
  \end{bmatrixfn}
  \begin{bmatrixfn}
    \bU(t)\\
    \bT(t)
  \end{bmatrixfn} +
  \begin{bmatrixfn}
    \bzero\\
    \eta_1(\Delta t)\bI
  \end{bmatrixfn}
  \bu(t),\quad
  \bQ=
  \begin{bmatrixfn}
    \eta_2^2\bI & \bzero\\
    \bzero &\eta_3^2\bI
  \end{bmatrixfn},\quad
             \bR=\bI,
\end{align*}
where $\bL=\{k_{ij}\}_{i\in\cV,j\in\cV}$ is the weighted graph Laplacian. Note that increasing $N$ expands the domain rather than refining the spatial discretization. One can see that the system matrices are uniformly upper bounded and $\bR$ is uniformly lower bounded.
{
  Further, $(\bA,\bB)$ is stabilizable and $(\bA,\bQ^{1/2})$ is detectable with uniform constants, since one can make $\bA-\bB\bK$ and $\bA-\bK'\bQ^{1/2}$ nilpotent (thus their spectral radius is $0$) by choosing
  \begin{align*}
    \bK&:=
    \begin{bmatrixfn}
      \frac{1}{\eta_1(\Delta t)^2}\bI & (2\bI-(\Delta t)\bL )/\eta_1(\Delta t)\\
    \end{bmatrixfn},\;
    \bK':=
    \begin{bmatrixfn}
      (1/|\eta_2|) \bI + (\bI-(\Delta t)\bL)/|\eta_2| & \bzero\\
      (1/(\Delta t)|\eta_2|)(\bI-(\Delta t)\bL)^2 &\bzero\\
    \end{bmatrixfn}.
  \end{align*}
  That is, it is sufficient for $\eta_1,\eta_2\neq 0$ to satisfy Assumption \ref{ass:main}. \hfill$\blacksquare$
}

{\it Example} \ref{eg:freq} (revisited). The frequency control problem can be analyzed in a similar manner. The model can be expressed by
  \begin{align*}
    \begin{bmatrixfn}
      \btheta(t+1)\\
      \bomega(t+1)
    \end{bmatrixfn}=
    \begin{bmatrixfn}
      \bI&(\Delta t)\bI\\
      - (\Delta t)\bL&\bI \\
    \end{bmatrixfn}
    \begin{bmatrixfn}
      \btheta(t)\\
      \bomega(t)
    \end{bmatrixfn} +
    \begin{bmatrixfn}
      \bzero\\
      \eta_1(\Delta t)\bI
    \end{bmatrixfn}
    \bu(t),\quad
    \bQ=
    \begin{bmatrixfn}
      \eta_2^2\bI& \bzero\\
      \bzero&\eta_3^2\bI
    \end{bmatrixfn},\quad
        \bR=\bI,
  \end{align*}
  where $\bL=\{k_{ij}\}_{i\in\cV,j\in\cV}$ is the weighted graph Laplacian.
  {
    Similarly to Example \ref{eg:hvac}, the system matrices are uniformly upper bounded and $\bR$ is uniformly lower bounded.
  Moreover, $(\bA,\bB)$ is stabilizable and $(\bA,\bQ^{1/2})$ is detectable with uniform constants, since one can make $\bA-\bB\bK$ and $\bA-\bK'\bQ^{1/2}$ nilpotent by choosing
  \begin{align*}
    \bK&:=
    \begin{bmatrixfn}
      \frac{1}{\eta_1(\Delta t)^2}\bI  - \frac{1}{\eta_1}\bL& \frac{2}{\eta_1(\Delta t)}\bI\\
    \end{bmatrixfn},\;
    \bK':=
    \begin{bmatrixfn}
      (2/|\eta_2|) \bI &\bzero\\
      (1/(\Delta t)|\eta_2|) \bI - ((\Delta t)/|\eta_2|)\bL&\bzero\\
    \end{bmatrixfn}.
  \end{align*}
  That is, it is sufficient for $\eta_1,\eta_2\neq 0$ to satisfy Assumption \ref{ass:main}. \hfill$\blacksquare$
}
\section{Sufficient Conditions for Uniform Stabilizability and Detectability}\label{sec:uniform}
As discussed in Remark \ref{rmk:uniform}, sufficient conditions for the uniformity of Assumption \ref{ass:main} are of interest. While ensuring the uniformity for the boundedness (Assumption \ref{ass:main}\ref{ass:bdd}) and positive definiteness of $\bR$ (Assumption \ref{ass:main}\ref{ass:cvx}) is straightforward, {ensuring it for} {$(\bA,\bB)$-stabilizability} (Assumption \ref{ass:main}\ref{ass:stab}) and {$(\bA,\bQ^{1/2})$-detectability} (Assumption \ref{ass:main}\ref{ass:dect}) is nontrivial. We first revisit Examples \ref{eg:hvac} and \ref{eg:freq} and see under which conditions the uniformity of stabilizability and detectability may be violated. 

{{\it Example} \ref{eg:hvac} (revisited).} Consider a modification of the model in Example \ref{eg:hvac}, where the actuators are installed only in a subset $\cV_c$ of nodes:
\begin{align}\label{eqn:hvac-mod}
  \dot{T}_i = - \sum_{j\in\cN_{\cG}[i]} k_{ij}(T_i-T_j)  +
  \left\{
  \eta_1 u_i\text{ if }i\in\cV_c,\;
  0\text{ otherwise}
  \right\}
  .
\end{align}
One can observe that the analysis in Section \ref{sec:main} does not apply anymore, and it is nontrivial to see under which conditions Assumption \ref{ass:main}\ref{ass:stab} is satisfied. \hfill$\blacksquare$

{{\it Example} \ref{eg:freq} (revisited).} A similar situation may arise in the detectability side. Consider a modification of the performance index in Example \ref{eg:freq}, where only the states in a subset $\cV_o$ of nodes are observed by the performance index:
\begin{align}\label{eqn:freq-mod}
  \ell_i(\theta_i,\omega_i,u_i)=\left\{\eta_2^2\theta^2_i \text{ if }i\in\cV_o,\;0\text{ otherwise}\right\} + u^2_i.
\end{align}
Similarly, the satisfaction of Assumption \ref{ass:main}\ref{ass:dect} is not immediately clear. \hfill$\blacksquare$

As can be seen in these two examples, the satisfaction of the uniform stabilizability and detectability may not be immediately clear. In many practical applications, the network is often underactuated or undersensed. Thus, one must study sufficient conditions for uniform stabilizability and detectability.

We now discuss sufficient conditions for Assumption \ref{ass:main}. {To facilitate the discussion, we introduce some new notation: $\bx_{\cV'}:=\{x_i\}_{i\in\cV'}$; $\bu_{\cV'}:=\{u_i\}_{i\in\cV'}$; $ \bA_{\cV',\cV''}:= \{A_{ij}\}_{i\in\cV',j\in\cV''}$; $ \bB_{\cV',\cV''}:= \{B_{ij}\}_{i\in\cV',j\in\cV''}$; $ \bQ_{\cV',\cV''}:= \{Q_{ij}\}_{i\in\cV',j\in\cV''}$; $ \bR_{\cV',\cV''}:= \{R_{ij}\}_{i\in\cV',j\in\cV''}$, where $\cV',\cV''\subseteq \cV$.}
\begin{assumption}\label{ass:uniform}
  $\exists L_0>1$, $\gamma_0,\alpha_0\in(0,1)$, and a partition $\{\cV_k\}_{k\in\cK}$ of $\cV$ such that
  \begin{enumerate}[label=(\alph*)]
  \item\label{ass:uniform-blkdiag} {$\forall k,k'(\neq k)\in\cK$, $\bQ_{\cV_k,\cV_{k'}}=\bzero$, $\bR_{\cV_k,\cV_{k'}}=\bzero$, $\bB_{\cV_k,\cV_{k'}}=\bzero$.}
  \item\label{ass:uniform-bdd} {$\forall k,k'\in\cK$, $\|\bQ_{\cV_k,\cV_{k'}}\|,\|\bR_{\cV_k,\cV_{k'}}\|,\|\bA_{\cV_k,\cV_{k'}}\|,\|\bB_{\cV_k,\cV_{k'}}\|\leq L_0$.}
  \item\label{ass:uniform-cvx} {$\forall k\in\cK$, $\bR_{\cV_k,\cV_k}\succeq \gamma_0 \bI$.}
  \item\label{ass:uniform-stab} $\forall k,k'(\neq k)\in\cK$, $(\bA_{\cV_k,\cV_k},\bB_{\cV_k,\cV_k})$ is $(L_0,\alpha_0)$-stabilizable, and $\exists \obK_{\cV_k,\cV_{k'}}$ such that $\|\obK_{\cV_k,\cV_{k'}}\|\leq L_0$ and $\bA_{\cV_k,\cV_{k'}}=\bB_{\cV_k,\cV_k}\obK_{\cV_k,\cV_{k'}}$.
  \item\label{ass:uniform-dect} $\forall k,k'(\neq k)\in\cK$, $\bQ_{\cV_k,\cV_k}\succeq \bzero$, $(\bA_{\cV_k,\cV_k},\bQ^{1/2}_{\cV_k,\cV_k})$ is $(L_0,\alpha_0)$-detectable, and $\exists \obK'_{\cV_{k'},\cV_k}$ such that $\|\obK'_{\cV_{k'},\cV_k}\|\leq L_0$ and $\bA_{\cV_{k'},\cV_k}=\obK'_{\cV_{k'},\cV_k} \bQ^{1/2}_{\cV_k,\cV_k}$.
  \end{enumerate}
\end{assumption}
{Assumption \ref{ass:uniform}\ref{ass:uniform-blkdiag} requires that the system can be decomposed into blocks that do not have shared objective or actuators; that is, the coupling between the blocks is only made through dynamic mapping $\bA$. Assumptions \ref{ass:uniform}\ref{ass:uniform-bdd} and \ref{ass:uniform}\ref{ass:uniform-cvx} require that the blocks have bounded system matrices and positive definite $\bR$ blocks.}
Assumption \ref{ass:uniform}\ref{ass:uniform-stab} requires that the system can be partitioned into uniformly stabilizable blocks and that the effect of adjacent blocks can be rejected in one step. In particular, the effect of $\bx_{\cV_{k'}}$ on the dynamics of system $k$ can be captured by $\bA_{\cV_k,\cV_{k'}}\bx_{\cV_{k'}}$, which can be canceled out by adding a $\overline{\bK}_{\cV_k,\cV_{k'}} \bx_{\cV_{k'}}$ term to $\bu_{\cV_k}$. {In order for this to be true, $\bA$ should exhibit sparse connections between the blocks induced by $\{\cV_{k}\}_{k\in\cV}$.} Similarly, Assumption \ref{ass:uniform}\ref{ass:uniform-dect} requires that the system can be partitioned into uniformly detectable blocks and that the effect of the adjacent blocks can be filtered in one step. We also state an additional assumption on the topology of the partition. 

\begin{assumption}\label{ass:degree}
  {Given the partition} $\{\cV_k\}_{k\in\cK}$ of $\cV$, $\exists D\in\mathbb{I}_{\geq 0}$ such that
  \begin{align*}
    |\{k'\in\cK:\exists \{i,j\}\in\cE \text{ such that }i\in\cV_k,j\in\cV_{k'}\}|\leq D,\quad  \forall k\in\cK.
  \end{align*}
\end{assumption}
This assumption says that the blocks in the partition $\{\cV_k\}_{k\in\cK}$ have a bounded number of neighboring blocks. We are now ready to state the main theorem of Section \ref{sec:uniform}.
\begin{theorem}\label{thm:uniform}
  Under $\alpha:=\alpha_0$, $\gamma:=\gamma_0$, and $L:=L_0 D$, the following holds.
  \begin{enumerate}[label=(\alph*)]
  \item\label{thm:uniform-bdd} Assumptions \ref{ass:uniform}\ref{ass:uniform-blkdiag}, \ref{ass:uniform}\ref{ass:uniform-bdd}, and \ref{ass:degree} imply Assumption \ref{ass:main}\ref{ass:bdd}.
  \item\label{thm:uniform-cvx} Assumptions \ref{ass:uniform}\ref{ass:uniform-blkdiag} and \ref{ass:uniform}\ref{ass:uniform-cvx} imply Assumption \ref{ass:main}\ref{ass:cvx}.
  \item\label{thm:uniform-stab} Assumptions \ref{ass:uniform}\ref{ass:uniform-blkdiag}, \ref{ass:uniform}\ref{ass:uniform-bdd}, \ref{ass:uniform}\ref{ass:uniform-stab}, and \ref{ass:degree} imply Assumption \ref{ass:main}\ref{ass:stab}.
  \item\label{thm:uniform-dect} Assumptions \ref{ass:uniform}\ref{ass:uniform-blkdiag}, \ref{ass:uniform}\ref{ass:uniform-bdd}, \ref{ass:uniform}\ref{ass:uniform-dect}, and \ref{ass:degree} imply Assumption \ref{ass:main}\ref{ass:dect}.
  \end{enumerate}
  Therefore, under Assumptions \ref{ass:poly}, \ref{ass:uniform}, and \ref{ass:degree}, the results of Theorems \ref{thm:decay}, \ref{thm:stab}, and \ref{thm:regret} and Corollary \ref{cor:decay} hold.
\end{theorem}
The proof is in Appendix \ref{apx:uniform}. The proof follows from the simple observation that if the effect of coupling can be rejected via the feedback/filter, then the decoupled subsystems are uniformly stabilizable/detectable.

Theorem \ref{thm:uniform} implies that the conditions in Assumption \ref{ass:main} can be made uniform by making comparable uniform assumptions that apply to {\it each block of nodes}. For example, instead of assuming the boundedness of the whole system, Assumption \ref{ass:uniform}\ref{ass:uniform-bdd} makes a boundedness assumption for each block. Similarly, rather than assuming that the whole system is stabilizable, we require that each block be stabilizable and capable of rejecting the coupling in one step.

We now discuss under which conditions Assumption \ref{ass:uniform} holds. Consider the underactuated HVAC system in \eqref{eqn:hvac-mod} and assume that the system is {\it assembled} by connecting uniformly stabilizable clusters of zones and that the connections are {\it made only between controlled zones} $\cV_c$ {(recall that $\cV_c$ is the set of nodes with actuators \eqref{eqn:hvac-mod})}. Then, one can see that each cluster remains uniformly stabilizable, and the effects of coupling between the clusters can be rejected by the actuators at the point of coupling.  Consequently, the assembled system is uniformly stabilizable. Similarly, for the undersensed power system in \eqref{eqn:freq-mod}, as long as the system is assembled by uniformly detectable clusters of buses and the connections are established only via the observed buses, the system remains uniformly detectable. From this observation, we can derive a design principle for the large-scale system that facilitates decentralization: The system must be decomposable to stabilizable/detectable blocks, and each block must have the ability to reject/filter the effect of coupling. 

\section{Numerical Results}\label{sec:num}
In this section we demonstrate the theoretical developments in Corollary \ref{cor:decay} and Theorem \ref{thm:regret} with Examples \ref{eg:hvac} and \ref{eg:freq}. For Example \ref{eg:hvac}, we assume that the graph is a $10\times 10$ mesh, $\Delta t= 1$s, $k_{ij} = 0.05$. For Example \ref{eg:freq}, we use the graph topology and susceptance data in the IEEE 118-bus case. We set $\Delta t= 5\times 10^{-6}$s, $k_{ij}=B_{ij}V_{\text{ref}}/M^2$, where $B_{ij}$ is the line susceptance, $M=10^{-5}\text{kg}\text{m}^2$, and $V_{\text{ref}}=132$kV. We use the original models in Examples \ref{eg:hvac} and \ref{eg:freq}, where the actuators are installed in every node and all the states are observed via the performance index. We set $\eta_1=\eta_2=\eta$ and $\eta_3=0$, where $\eta$ will be varied. The results can be reproduced with the source code available at \url{https://github.com/sshin23/near-optimal-distributed-lqr}.

We compare the relative truncation error $\|\bK^\kappa-\bK^\star\| / \|\bK^\star\|$ and the relative optimality gap $\|\bP^\kappa-\bP^\star\|/\|\bP^\star\|$. Here $\bP^*$ ($*=\kappa$ or $\star$) is the solution of the discrete Lyapunov equation: $(\bPhi^*)^\top  \bP \bPhi^* - \bP +  \bQ + (\bK^*)^\top \bR\bK^* = \bzero$. Note that $(1/2)\obx^\top \bP^*\obx$  is the cost-to-go function associated with gain $\bK^*$. Accordingly, we have that $\|\bP^\star-\bP^\kappa\| = 2\max\{ J(\bpi^\kappa(\cdot);\bx) - J(\bpi^\star(\cdot);\bx):\|\bx\|\leq 1\}$. Therefore, $\|\bP^\kappa-\bP^\star\|/\|\bP^\star\|$ represents the worst-case relative optimality gap. Throughout the case study, we vary two parameters: $\eta$ and $\kappa$. First, as $\kappa$ becomes large, we expect the relative truncation error and relative optimality gap to approach  $0$ (Corollary \ref{cor:decay} and Theorem \ref{thm:stab}). Furthermore, since the system becomes closer to violating the stabilizability and detectability conditions as $\eta\rightarrow 0$, we expect that the decay of the relative truncation error and optimality gap becomes more pronounced when $\eta$ is large. From the results in Figure \ref{fig:num}, we can confirm that the relative truncation error and optimality gap decay exponentially with $\kappa$ except for high $\kappa$ and $\eta$ cases, where finite numerical precision might have caused issues. Furthermore, the decay rate becomes faster as $\eta$ increases, as expected. Therefore, our numerical results verify our theoretical findings.

\begin{figure}[t!]
  \centering
  \includegraphics[width=.49\textwidth]{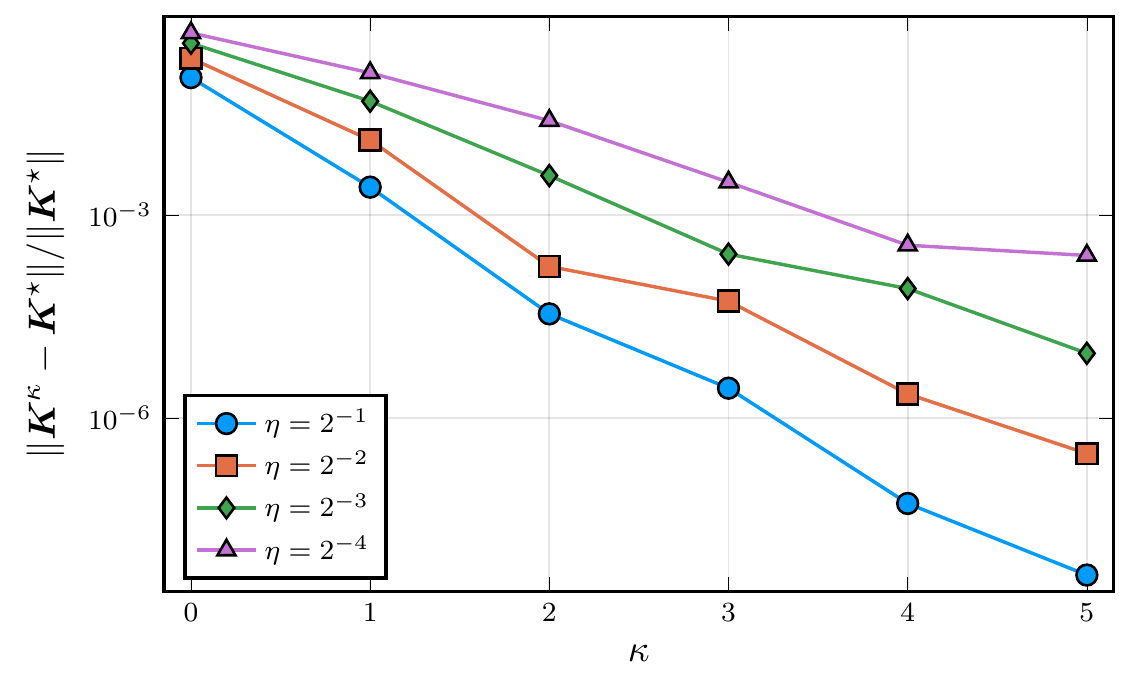}
  \includegraphics[width=.49\textwidth]{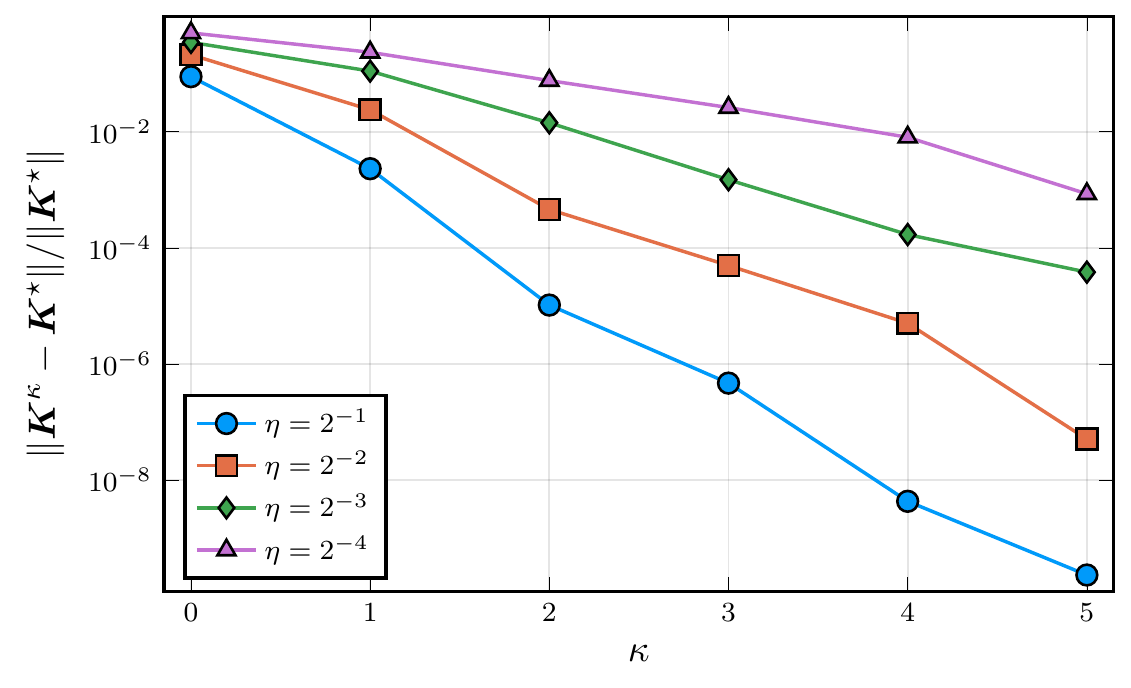}
  \includegraphics[width=.49\textwidth]{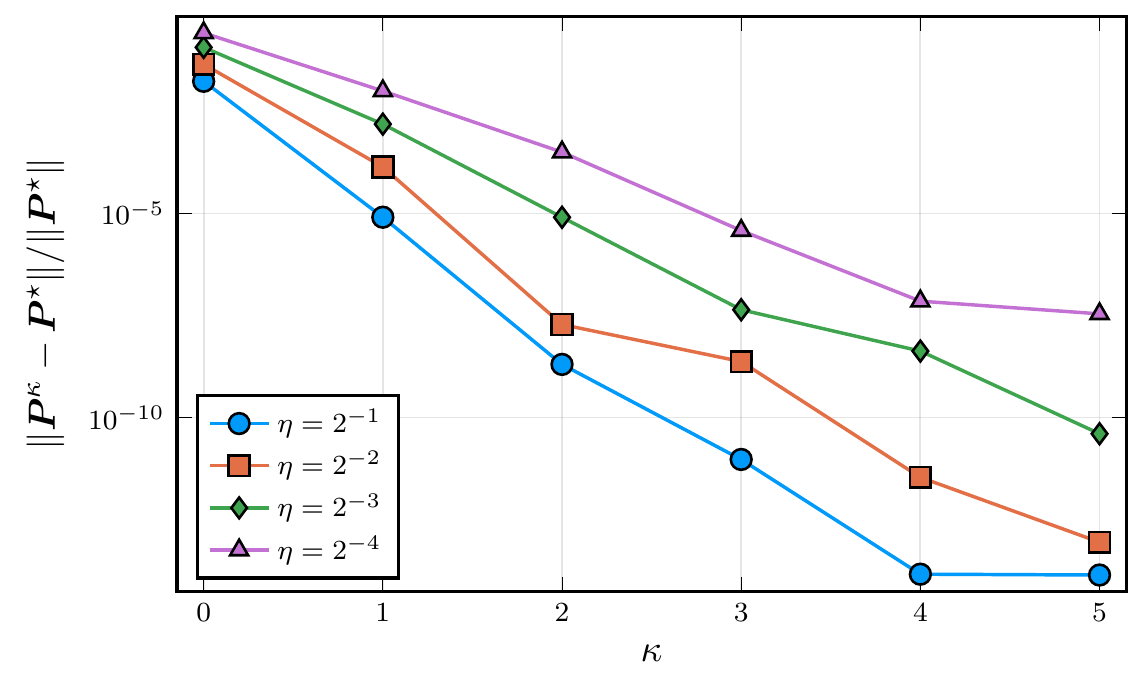}
  \includegraphics[width=.49\textwidth]{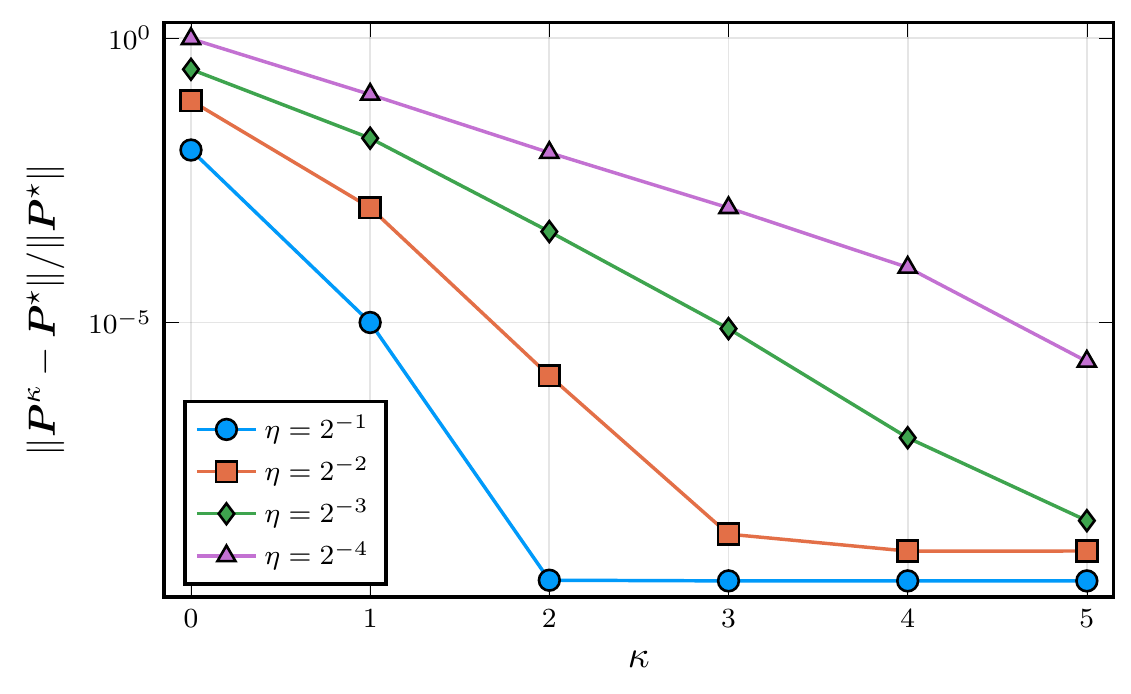}
  \caption{Numerical results. HVAC system control (left) and power system frequency control (right). Relative truncation error (top) and relative optimality gap (bottom).}
  \label{fig:num}
\end{figure}

\section{Conclusions}\label{sec:con}
{We have presented a stability and performance analysis of $\kappa$-distributed LQR.} This controller becomes more centralized as $\kappa$ becomes large and more decentralized as $\kappa$ becomes small. We have shown that under mild assumptions,  the $\kappa$-distributed controller is stabilizing for sufficiently large $\kappa$ and the optimality gap exponentially decays with $\kappa$. This result manifests the trade-off between decentralization and nominal performance; the control performance exponentially improves as the degree of decentralization is reduced. Thus, with a moderate degree of decentralization, distributed control can achieve {\it near-optimal performance}. Consequently, this result guides the design of distributed controllers that can balance computation loads, privacy, and performance. {As a future study, we are interested in analyzing the performance of $\kappa$-distributed control under more general settings, such as continuous-domain, nonlinear, constrained, and stochastic optimal control settings.}

{
  \subsection*{Acknowledgments}
  We are grateful to the anonymous referees, whose comments  greatly improved the paper.
}

\appendix
\section{Preliminaries}\label{apx:prelim}
This section presents the preliminary results for showing Theorems \ref{thm:decay}, \ref{thm:stab}. In Sections \ref{sec:inv}, \ref{sec:reg} we summarize the key results in \cite{shin2022exponential}. In Section \ref{sec:sys} we specialize the results in Sections \ref{sec:inv}, \ref{sec:reg} to finite-horizon LQR. In Section \ref{sec:dare} we present the bound result for the solutions of DARE. In Section \ref{sec:opt} we show the stability of the optimal LQR.

\subsection{Graph-Induced Banded Matrix Properties}\label{sec:inv}
First, we introduce the notion of graph-induced bandwidth. Recall that $\mathbb{I}$ denotes the set of integers and $\mathbb{I}_A := \mathbb{I} \cap A$ for any set $A$.
\begin{definition}[Graph-Induced Bandwidth]\label{def:bandwidth}
  Consider $\bH\in\mathbb{R}^{\bm_{\bH}\times \bn_{\bH}}$, a graph $\cG:=(\cV,\cE)$, and $\cI:=\{I_i\}_{i\in \cV}$, $\cJ:=\{J_i\}_{i\in \cV}$ that partition $\mathbb{I}_{[1,\bm_{\bH}]}$, $\mathbb{I}_{[1,\bn_{\bH}]}$, respectively. We say $\bH$ has bandwidth $B$ induced by $(\cG,\cI,\cJ)$ if $B$ is the smallest nonnegative integer satisfying $\bH_{[i][j]}:=\bH[I_i,J_j]=\bzero$ for any $i,j\in \cV$ with $d_\cG(i,j)> B$.
\end{definition}

We say a matrix is {\it graph-induced banded} if the matrix has a small bandwidth induced by a certain graph and index sets. A graph-induced banded matrix has the property that the $(i,j)$-th block of its inverse decays exponentially in the distance between $i$ and $j$. The following theorem establishes such a result.

\begin{theorem}\label{thm:inv}
  Consider $\bH\in\mathbb{R}^{\bn_{\bH}\times \bn_{\bH}}$, whose bandwidth is not greater than $1$, induced by $(\cG:=(\cV,\cE),\cJ:=\{J_i\}_{i\in \cV},\cI:=\{I_i\}_{i\in \cV})$; further, assume that $L_{\bH},\gamma_{\bH}>0$ satisfy $\gamma_{\bH}\leq \sigma(\bH) \leq L_{\bH}$ for all singular values of $\bH$. Then, $\Vert (\bH^{-1})_{[i][j]}\Vert \leq\Upsilon\rho^{d_\cG(i,j)}$ for $i,j\in \cV$, where $(\bH^{-1})_{[i][j]}:=(\bH^{-1})[I_i,J_j]$ and $\rho,\Upsilon$ are defined in \eqref{eqn:decay-constants}.
\end{theorem}
The proof is available in \cite{shin2022exponential}. Since $\rho\in(0,1)$, Theorem \ref{thm:inv} says that the $(i,j)$-th block of $\bH^{-1}$ decays exponentially with respect to $d_\cG(i,j)$, and the decay rate $\rho$ becomes faster as the condition number of $\bH$ becomes small. Note that one can use different graph and index sets for imposing structure on the same matrix. Later we will see that one can characterize the sparsity structure of the KKT matrix of finite-horizon LQR either with a temporal graph (Figure \ref{fig:space-time-graph}, top), where the spatial coupling is collapsed into a single node, or with a space-time graph, where each node represents a particular agent in the space at a particular time (Figure \ref{fig:space-time-graph}, bottom).

\subsection{Uniform Regularity}\label{sec:reg}
Theorem \ref{thm:inv} says that we need explicit upper and lower bounds of the singular values of $\bH$ in order to bound $\Upsilon,\rho$. We now  focus on the matrices that derive from the KKT conditions of equality-constrained quadratic programs by assuming
$\bH:=
\begin{bmatrixfn}
  \bG&\bF^\top\\
  \bF
\end{bmatrixfn}$,
where $\bG$ is symmetric. The following theorem bounds the singular values of $\bH$; the sufficient conditions here are called {\it uniform regularity} conditions.

\begin{theorem}\label{thm:uniform-regularity}
  Suppose that there exist $L_{\bH},\gamma_{\bG},\gamma_{\bF}> 0$ such that
  \begin{align}\label{eqn:uniform-regularity}
    \|\bH\|\leq L_{\bH},\quad ReH(\bG,\bF)\succeq \gamma_{\bG} \bI,\quad \bF \bF^\top \succeq \gamma_{\bF} I.
  \end{align}
  Here $ReH(\bG,\bF):=\bZ^\top \bG \bZ$, where $\bZ$ is {a null space matrix} for $\bF$. 
  Then $\gamma_{\bH}\leq \sigma(\bH)\leq L_{\bH}$ for all singular values of $\bH$, where $\gamma_{\bH}$ is defined in \eqref{eqn:decay-constants}.
\end{theorem}
\begin{proof}
  The upper bound follows from that the singular values are upper bounded by the induced 2-norm. By \cite[Lemma 5.6]{shin2022exponential}, we have that $(\gamma_{\bG}/2)\bI\preceq \bT\preceq L_{\bH}(1+\overline{\mu} L_{\bH})$ for $\bT:=\bG+\overline{\mu}\bF^\top \bF$ and $\overline{\mu}$ given in \eqref{eqn:decay-constants}. We observe that:
  \begin{align*}
      \bH^{-1}=
    \begin{bmatrixfn}
      \bT^{-1} - \bT^{-1}\bF^\top (\bF \bT^{-1}\bF^\top)^{-1} \bF \bT^{-1} & \bT^{-1}\bF^\top (\bF \bT^{-1}\bF^\top)^{-1}\\
      (\bF \bT^{-1}\bF^\top)^{-1} \bF \bT^{-1} &  \overline{\mu} \bI-(\bF \bT^{-1}\bF^\top)^{-1}\\
    \end{bmatrixfn},
  \end{align*}
  and accordingly, the lower bound is obtained by:
  \begin{align*}
    \|\bH^{-1}\|
    &\leq\|\bT^{-1}\| + \|\bT^{-1}\bF^\top (\bF \bT^{-1}\bF^\top)^{-1} \bF \bT^{-1}\| + 2 \| \bT^{-1}\bF^\top (\bF \bT^{-1}\bF^\top)^{-1}\|\\
    &\qquad+ \|\overline{\mu} \bI\|+\|(\bF \bT^{-1}\bF^\top)^{-1}\|\\
    &\leq\|\bT^{-1}\| + \left(1+2\|\bF\bT^{-1}\|+\|\bF\bT^{-1}\|^2 \right)\frac{1}{\ulambda(\bF \bT^{-1}\bF^\top)} + \overline{\mu}\bI \leq \gamma_{\bH}^{-1}.
  \end{align*}
\end{proof}
Theorem \ref{thm:uniform-regularity} suggests that under  uniform regularity, we have the desired boundedness of the singular values of $\bH$, which in turn guarantees that the decay bounds in Theorem \ref{thm:inv} are bounded. 

\subsection{Stabilizability and Detectability Imply Uniform Regularity}\label{sec:sys}
Now we aim to establish the relation between uniform regularity and system-theoretical properties (in particular, stabilizability and detectability). We first further specialize the setting to finite-horizon LQR. Consider the following.
\begin{subequations}\label{eqn:lqr-finite}\small
  \begin{align}
    \min_{\{\bx(t)\}_{t=0}^T,\{\bu(t)\}_{t=0}^{T-1}}\; &(1/2)\bx(T)^\top \bP\bx(T)+ \sum_{t=0}^{T-1} (1/2)\bx(t)^\top\bQ\bx(t)+(1/2)\bu(t)^\top\bR\bu(t)  \\
    \st\;&\bx(0)=\obx \quad(\blambda(0))\label{eqn:lqr-finite-con-1}\\
                                 &\bx(t+1)=\bA\bx(t) +\bB\bu(t),\;\forall t=0,1,\cdots,T-1.\quad(\blambda(t+1))\label{eqn:lqr-finite-con-2}
  \end{align}
\end{subequations}
Here $\{\blambda(t)\}_{t=0}^T$ are the dual variables associated with constraints \eqref{eqn:lqr-finite-con-1}--\eqref{eqn:lqr-finite-con-1}. We let $\bG\in\mathbb{R}^{(\bn_{\bx}(T+1)+\bn_{\bu} T) \times (\bn_{\bx}(T+1)+\bn_{\bu} T)}$ and $\bF\in\mathbb{R}^{\bn_{\bx}(T+1)\times (\bn_{\bx}(T+1)+\bn_{\bu} T)}$ be the objective Hessian and constraint Jacobian of \eqref{eqn:lqr-finite}:  
\begin{align}\label{eqn:bGF}
  \bG:=
  {\footnotesize\begin{bmatrixfn}
    \bQ\\
    &\bR\\
    &&\ddots\\
    &&&\bQ\\
    &&&&\bR\\
    &&&&&\bP\\
  \end{bmatrixfn}},\;
  \bF:=
  {\footnotesize\begin{bmatrixfn}
      \bI\\
      -\bA&-\bB&\bI\\
      &&-\bA&-\bB&\bI\\
      &&&&\ddots\\
      &&&&-\bA&-\bB&\bI\\
      &&&&&&-\bA&-\bB&\bI\\
    \end{bmatrixfn}}.
\end{align}
 Note that $\bG,\bF$, and $\bH$ depend on the horizon length $T$, but we suppress the dependency to reduce the notational burden. We now establish the relation between Assumption \ref{ass:main} and uniform regularity. 
\begin{theorem}\label{thm:ctrb}
  Under Assumption \ref{ass:main} and $\bQ \preceq \bP\preceq L_{\bP}\bI$, \eqref{eqn:uniform-regularity} holds for $L_{\bH},\gamma_{\bG},\gamma_{\bF}$ defined in \eqref{eqn:decay-constants}.
\end{theorem}

We first present a helper lemma, whose proof is available in \cite{shin2022exponential}.
\begin{lemma}\label{lem:sqrt}
  The following holds for $\bM:=\{M_{ij}\}_{i,j\in\cV}$:
  \begin{align*}
    \|\bM\|\leq \Big(\max_{i\in \cV}\sum_{j\in \cV}\Vert M_{ij}\Vert\Big)^{1/2}\Big(\max_{j\in \cV}\sum_{i\in \cV}\Vert M_{ij}\Vert\Big)^{1/2}.
  \end{align*}
\end{lemma}

Now we come back to the proof of Theorem \ref{thm:ctrb}. We will establish $\|\bH\| \leq L_{\bH}$, $\bF\bF^\top \succeq \gamma_{\bF} \bI$, and $ReH(\bG,\bF)\succeq \gamma \bI $ in this order.

\begin{proof}[Proof of $\|\bH\|\leq L_{\bH}$]
  Let $I_t$ be the index set for $[\bx(t);\bu(t);\blambda (t)]$. We observe
  \begin{align*}
    \bH[I_t,I_{t'}]=
    \left\{
    \begin{footnotesize}
      \begin{aligned}
        &\begin{bmatrixfn}
          \bQ&&\bI\\
          &\bR&\\
          \bI
        \end{bmatrixfn},
        \text{ if } t=t'\neq T;\quad
        \begin{bmatrixfn}
          &&\bzero\\&&\bzero\\
          -\bA&-\bB&
        \end{bmatrixfn},
        \text{ if } t-t'= 1;\quad
        \bzero \text{ if } t > t'+1 
        \\
        &
        \begin{bmatrixfn}
          \bP&\bI\\
          \bI
        \end{bmatrixfn},
        \text{ if } t=t'=T;\quad
        \begin{bmatrixfn}
          &&\bzero\\
          -\bA&-\bB&
        \end{bmatrixfn},
        \text{ if } t=T,t'= T-1
      \end{aligned}
    \end{footnotesize}
               \right\}.
  \end{align*}
  Because of the symmetry, only blocks with $t\geq t'$ are denoted. From Assumption \ref{ass:main}\ref{ass:bdd}, $\bP\preceq L_{\bP}\bI$, and Lemma \ref{lem:sqrt}, we obtain $\|\bH\|\leq \max(2L+1,L_{\bP}+1)$.
\end{proof}
\begin{proof}[Proof of $  \bF\bF^\top \succeq \gamma_{\bF} \bI$]
  We consider the following column operation:
  \begin{align}\label{eqn:F}
    \footnotesize\begin{bmatrixfn}
      \bI\\
      -\bA&-\bB&\bI\\
      &&\ddots\\
      &&-\bA&-\bB&\bI\\
    \end{bmatrixfn}
    \begin{bmatrixfn}
      \bI\\
      -\obK&\bI\\
      &&\ddots\\
      &&-\obK&\bI\\
      &&&&\bI\\
    \end{bmatrixfn}
    =
    \begin{bmatrixfn}
      \bI\\
      -\obPhi& - \bB&\bI\\
      &&\ddots\\
      &&&-\obPhi& - \bB&\bI\\
    \end{bmatrixfn},
  \end{align}
  where $\obPhi:=\bA-\bB\obK$ and $\obK$ is the $(L,\alpha)$-stabilizing feedback gain (from Assumption \ref{ass:main}\ref{ass:stab}). This indicates $\bF\bF^\top \succeq \ulambda(\bM_2\bM_2^\top)\bM_1\bM_1^\top \succeq \|\bM_1^{-1}\|^{-2}\|\bM_2^{-1}\|^{-2}\bI$, where
  \begin{align*}
    \bM_1
    &:=
      \begin{bmatrixfn}
        \bI\\
        -\obPhi & \bI\\
        &&\ddots\\
        &&-\obPhi & \bI\\
      \end{bmatrixfn};\;
    \bM_1^{-1}=
    \begin{bmatrixfn}
      \bI\\
      \obPhi & \bI\\
      \vdots&&\ddots\\
      \obPhi^{T}&\cdots&\obPhi & \bI\\
    \end{bmatrixfn};\;
    \bM_2
    :=
    \begin{bmatrixfn}
      \bI\\
      -\obK&\bI\\
      &&\ddots\\
      &&-\obK&\bI\\
      &&&&\bI\\
    \end{bmatrixfn}^{-1}.
  \end{align*}
  {The first inequality follows from (i) $\bM'\bM(\bM')^\top \succeq \ulambda(\bM)\bM'(\bM')^\top$ for $\bM\succeq \bzero$ and (ii) the observation that $\bM_1$ is a submatrix of the matrix on the right-hand side of \eqref{eqn:F}. The second inequality follows from the fact that $\bM \succeq \|\bM^{-1}\|^{-1}\bI$ for $\bM\succ \bzero$.}
  From $\obK\leq L$ and $(L,\alpha)$-stability of $\obPhi$, we obtain $\|\bM^{-1}_1\|\leq L/(1-\alpha)$ and $\|\bM^{-1}_2\|\leq 1+L$. Thus, we have $\bF\bF^\top \succeq (1-\alpha)^2/L^2(1+L)^2$.
\end{proof}

\begin{proof}[Proof of $ReH(\bG,\bF)\succeq \gamma \bI $]
  Since $\bP\succeq \bQ$, it suffices to prove for $\bP=\bQ$. Let
  \begin{align*}\footnotesize
    \bG_1:=
    \begin{bmatrixfn}
      \bQ\\
      &\bQ\\
      &&\ddots\\
      &&&\bQ\\
    \end{bmatrixfn},\;
    \bG_2:=
    \begin{bmatrixfn}
      \bR\\
      &\bR\\
      &&\ddots\\
      &&&\bR\\
    \end{bmatrixfn},\;
    \bF_1:=
    \begin{bmatrixfn}
      \bI\\
      -\bA&\bI\\
      &\ddots&\ddots\\
      &&-\bA&\bI\\
    \end{bmatrixfn},\;
    \bF_2:=
    \begin{bmatrixfn}
      \\
      -\bB\\
      &\ddots\\
      &&-\bB\\
    \end{bmatrixfn}.
  \end{align*}
  It suffices to show that for any pair of vectors $(\bX,\bU)$ satisfying $\bF_1\bX + \bF_2 \bU=\bzero$, we have $\bX^\top\bG_1 \bX + \bU^\top \bG_2\bU\geq \gamma_{\bG} (\|\bX\|^2+\|\bU\|^2)$.  We observe the following:
  \begin{align}\label{eqn:detect-1}
    \bX^\top\bG_1 \bX + \bU^\top \bG_2\bU
    &\geq \|\bG_3\bX\|^2 + (\gamma/2)\|\bU\|^2 + \gamma/2L^2 \|\bF_1\bX\|^2\nonumber \\
    &\geq(\gamma/2L^2) \ulambda\left(
      \begin{bmatrixfn}
        \bF^\top_1& -\bG^\top_3\\
      \end{bmatrixfn}
    \begin{bmatrixfn}
      \bF^\top_1& -\bG^\top_3\\
    \end{bmatrixfn}^\top
    \right)
    \|\bX\|^2
    +  (\gamma/2)\|\bU\|^2,
  \end{align}
  where $\bG_3:=\diag(\bQ^{1/2},\cdots,\bQ^{1/2})$. Here the first inequality follows from $\|\bF_2\|\leq L$, $\bG_2\succeq \gamma\bI$, $\bF_1 \bX + \bF_2\bU=\bzero$, and the second inequality follows from $\gamma/2L^2<1$ (from Assumption \ref{ass:main}).
  Also, $[\bF^\top_1\; -\bG^\top_3]$ can be permuted to
  \begin{align*}
    \begin{bmatrixfn}
      -\bQ^{1/2}&\bI\\
      &-\bA^\top&-\bQ^{1/2}&\bI\\
      &&&\ddots\\
      &&&-\bA^\top&-\bQ^{1/2}&\bI\\
    \end{bmatrixfn}.
  \end{align*}
  From the Assumption \ref{ass:main}\ref{ass:dect} and applying $\bF\bF^\top \succeq \gamma_{\bF} \bI$ while replacing $\bA\leftarrow \bA^\top $ and $\bB\leftarrow \bQ^{1/2}$, we obtain $[\bF^\top_1\; -\bG^\top_3][\bF^\top_1\; -\bG^\top_3]^\top\succeq\frac{(1-\alpha)^2 }{L^2 (1+L)^2}\bI$. From \eqref{eqn:detect-1} and $(1-\alpha)^2 /2L^4 (1+L)^2<1$ {(from Assumption \ref{ass:main})}, we obtain the desired result.
\end{proof}

\subsection{Boundedness of the Solutions of DARE}\label{sec:dare}
Consider the DARE:
\begin{align}\label{eqn:dare}
  \bP = \bA^\top \bP \bA - (\bA^\top\bP\bB) (\bR+\bB^\top \bP\bB)^{-1}(\bB^\top\bP\bA) + \bQ.
\end{align}
We aim to establish the boundedness of the solution of \eqref{eqn:dare}. 

\begin{theorem}\label{thm:dare}
  Under Assumption \ref{ass:main}, there is a unique solution $\bP^\star$ of \eqref{eqn:dare} satisfying $\bQ \preceq \bP^\star \preceq L_{\bP} \bI$, where $L_{\bP}$ is defined in \eqref{eqn:decay-constants}.
\end{theorem}

\begin{proof}
  The existence and uniqueness of the positive semi-definite solution follow from \cite[Theorem 2.4-2]{lewis2012optimal}. Recall that we use $\obx$ to denote the initial state of the centralized LQR problem \eqref{eqn:lqr}. The lower boundedness of $\bP^\star$ is trivial from the observation that $(1/2)\obx^\top \bP^\star\obx$ is the cost-to-go function, and the cost-to-go at $\obx$ is greater than $(1/2)\obx^\top\bQ\obx$ because of the first stage cost. Now we prove $\bP^\star\preceq L_{\bP}\bI$. From the observation that $(1/2)\obx^\top \bP^\star\obx$ is the cost-to-go function, it suffices to show that for any $\obx\in\mathbb{R}^{\bn_{\bx}}$, there exists a sequence $\{\bx(t),\bu(t)\}_{t\in\mathbb{I}_{\geq 0}}$ satisfying \eqref{eqn:lqr-con-1}-\eqref{eqn:lqr-con-2} and $\sum_{t=0}^\infty \bx(t)^\top\bQ\bx(t)+ \bu(t)^\top\bR\bu(t) \leq  L_{\bP} \|\obx\|^2$. We construct such a sequence by letting $\bu(t):=-\obK\bx(t)$, where $\obK$ is the $(L,\alpha)$-stabilizing feedback gain (from Assumption \ref{ass:main}\ref{ass:stab}). We can observe that $\|\bx(t)\| \leq \|(\bA-\bB\obK)^t \|\|\obx\|\leq L\alpha^t \|\obx\|$, $\|\bu(t)\|\leq \|\obK\|\|\bx(t)\|\leq L^2 \alpha^t\|\obx\|$. Thus, $\sum_{t=0}^\infty \bx(t)^\top\bQ\bx(t)+ \bu(t)^\top\bR\bu(t) \leq  \frac{L^3 ( 1+L^2)}{1-\alpha^2} \|\obx\|^2$. 
\end{proof}

\subsection{Exponential Stability of Optimal LQR}\label{sec:opt}
We now perform the stability analysis of the centralized LQR. In particular, we show the stability of $\bPhi^\star$. Recall that the state transition mapping for the optimal policy is defined as $\bphi^\star(x):= \bPhi^\star x$, where $\bPhi^\star := \bA - \bB \bK^\star$ and $\bK^\star$ is the optimal gain matrix. 
\begin{theorem}\label{thm:stab-0}
  Under Assumption \ref{ass:main}, $\bPhi^\star$ is $(\Upsilon,\rho)$-stable (defined in \eqref{eqn:decay-constants}).
\end{theorem}
\begin{proof}
  It suffices to show $\|\bx^\star(t)\|\leq \Upsilon\rho^t \|\obx\|$ for any $\obx\in\mathbb{R}^{\bn_{\bx}}$ and $t\in\mathbb{I}_{\geq 0}$ because $\bx^\star(t)=(\bPhi^\star)^t \obx$. Consider a finite-horizon LQR \eqref{eqn:lqr-finite} with $\bP=\bP^\star$ (the solution of the DARE \eqref{eqn:dare}). Since $\bP^\star$ is the Hessian of cost-to-go function, the solution of \eqref{eqn:lqr-finite} is equal to the corresponding piece of the infinite-horizon LQR solution. For given $t\in\mathbb{I}_{\geq 0}$, we choose some $T\geq t$. The KKT matrix $\bH$ of \eqref{eqn:lqr-finite} is a banded matrix induced by the temporal graph $\cG=(\cV:=\mathbb{I}_{[0,T]},\cE:=\{\{0,1\},\{1,2\},\cdots,\{T-1,T\})$ and index sets $\cI=\cJ:=\{I_t\}_{t=0}^T$ with $I_t$ holding the index of $[\bx(t);\bu(t);\blambda (t)]$ (Figure \ref{fig:space-time-graph}, top). The mapping $\obx\mapsto\bx^\star(t)$ can be obtained by $\Pi_{\bx}(\bH^{-1})_{[t][0]}\Pi^\top_{\blambda}$, where $(\bH^{-1})_{[t][0]}:=(\bH^{-1})[I_t,I_0]$, and $\Pi_{\bx}$ and $\Pi_{\blambda}$ are the projection mappings that extract $\bx$ and $\blambda$, respectively, from $(\bx,\bu,\blambda)$. {Here, $\Pi_{\blambda}$ appears, since the initial state $\obx$ enters into the constraint, whose index within $(\bx,\bu,\blambda)$ is associated with $\blambda$.} Thus, it suffices to show that $\|\bH^{-1}_{[t][0]}\|\leq \Upsilon\rho^{t}$. From Theorems \ref{thm:inv}, \ref{thm:uniform-regularity}, \ref{thm:ctrb}, and \ref{thm:dare} we obtain $\|\bH^{-1}_{[t][0]}\|\leq \Upsilon\rho^{t}$. 
\end{proof}

\section{Proofs}\label{apx:proofs}
\subsection{Proof of Theorem \ref{thm:decay}}\label{apx:decay}
Consider the space-time graph $\cG_{T}:=(\cV_T,\cE_{T})$ (Figure \ref{fig:space-time-graph}, bottom), where $\cV_T:=\mathbb{I}_{[0,T]}\times\cV$, $\cE_{T}:=\{\{(t,i),(t,j)\}\}_{\{i,j\}\in\cE,t\in\mathbb{I}_{[0,T-1]}}\cup\{\{(0,i),(1,i)\},\cdots,\{(t-1,i),(t,i)\}\}_{i\in\cV} \cup \{\{(T,i),(T,j)\}\}_{i,j\in\cV}$. Here, the pairs $(t,i)\in\mathbb{I}_{[0,T]}\times\cV$ constitute the space-time node set $\cV_T$; and the spatial couplings, temporal couplings, and the dense coupling at the $t=T$ (due to potentially dense $\bP$) constitute the edge set $\cE_T$. {The KKT matrix $\bH$ of \eqref{eqn:lqr-finite} has a bandwidth not greater than $1$ induced by $(\cG_{T},\cI_T,\cI_T)$ with $\cI_T:=\{I_{t,i}\}_{t\in\mathbb{I}_{[0,T]},i\in\cV}$ and $I_{t,i}$ holding the index of $(x_i(t),u_i(t),\lambda_i(t))$ because the sparsity structure of $\bA$, $\bB$, $\bQ$, and $\bR$ assumed in \eqref{eqn:settings} only allows the coupling between the directly neighbored nodes, and the temporal couplings only exist between adjacent time indices}. With $\bP=\bP^\star$ (the solution of the DARE \eqref{eqn:dare}), the solution of \eqref{eqn:lqr-finite} is equal to the corresponding piece of the solution of \eqref{eqn:lqr}. Thus, the solution mapping $\obx\mapsto\bu(0)$ of \eqref{eqn:lqr-finite} is equal to the optimal gain matrix $\bK^\star$, and this implies $ K^\star_{ij}= \Pi_{u_i}(\bH^{-1})_{[0,i][0,j]}  \Pi^\top_{\lambda_j}$, where $(\bH^{-1})_{[t,i][t,j]}:=\bH^{-1}[I_{t,i},I_{t,j}]$, and $\Pi_{u_i}$ and $\Pi_{\lambda_j}$ are projection operators extracting $u_i$ and $\lambda_j$ from $(x_i,u_i,\lambda_i)$ and $(x_j,u_j,\lambda_j)$, respectively. {Here, $\Pi_{\lambda_j}$ appears, since the initial state $\overline{x}_j$ enters into the constraint, whose index within $(x_j,u_j,\lambda_j)$ is associated with $\lambda_j$.} Thus, it suffices to show that $\| (\bH^{-1})_{[0,i][0,j]}\| \leq \Upsilon \rho^{d_\cG(i,j)}$. We now choose $T$ such that $T\geq \max_{i,j\in\cV}d_{\cG}(i,j) $. We observe that $d_{\cG_T}((0,i),(0,j)) = d_{\cG}(i,j)$, since the dense coupling at stage $T$ is sufficiently far away from stage $0$ {(for intuition, see Figure \ref{fig:space-time-graph}, bottom)}. By Assumption \ref{ass:main} and Theorems \ref{thm:inv}, \ref{thm:uniform-regularity}, \ref{thm:ctrb}, and \ref{thm:dare}, we have $\| (\bH^{-1})_{[0,i][0,j]}\| \leq \Upsilon \rho^{d_\cG(i,j)}$. Lastly, $\rho\in(0,1)$ and $\Upsilon\geq 1$ follow from {$\gamma_{\bH}< 1$ and $L_{\bH}> 1$, which can be verified from their definitions in \eqref{eqn:decay-constants} and Assumption \ref{ass:main}. }

\subsection{Proof of Corollary \ref{cor:decay}}\label{apx:cor}
  We observe $\sum_{j\in\cV\setminus\cN_{\cG}^\kappa[i]}\|K^\star_{ij}\|
  \leq \sum_{d=\kappa+1}^\infty \Upsilon p(d)\cdot(\rho/\delta)^d \delta^d
  \leq (\sup_{d\in\mathbb{I}_{\geq 0}} p(d) (\rho/\delta)^d)\frac{\Upsilon \delta}{1-\delta} \delta^{\kappa}$ for any $i\in\cV$. Here the first inequality follows from Assumption \ref{ass:poly}, and the second inequality follows from the fact that the product of a {subexponential} and an exponentially decaying functions is bounded above. Similarly, one can obtain the bound for $\sum_{i\in\cV\setminus\cN_{\cG}^\kappa[j]}\|K^\star_{ij}\|$. By Lemma \ref{lem:sqrt}, we obtain the desired result. Finally, $\delta\in(0,1)$ follows from $\rho\in(0,1)$ (Theorem \ref{thm:decay}).

\subsection{Proof of Theorem \ref{thm:stab}}\label{apx:stab}
For $V(\bx)  :=\sum_{t=0}^\infty \|(\bPhi^\star)^t\bx\|^2 = \bx^\top \bV\bx$, we have
  \begin{align}\label{eqn:lyapV}
    - V(\bx) +V(\bphi^\kappa(\bx))
    &= -\|\bx\|^2 - V(\bphi^\star(\bx)) + V(\bphi^\kappa(\bx))\nonumber\\
    &=  -\|\bx\|^2 + (-\bphi^\star(\bx)+ \bphi^\kappa(\bx))^\top\bV(\bphi^\star(\bx)+ \bphi^\kappa(\bx))\nonumber\\
    &= -\|\bx\|^2 + V(\bB(\bK^\star-\bK^\kappa)\bx) + 2\bphi^\star(\bx)^\top \bV (\bB(\bK^\star-\bK^\kappa)\bx) \nonumber\\
    &\leq -(1- \frac{\Upsilon^2\Psi L(L\Psi + 2L(1+L_{\bP}L^2/\gamma))}{1-\rho^2} \delta^\kappa) \|\bx\|^2,
  \end{align}
  where the equalities follow from the definitions of $V(\cdot)$ and $\bphi(\cdot)$, and the inequality follows from $\|\bx\|^2\leq V( \bx) \leq (\Upsilon^2 / (1-\rho^2))\|\bx\|^2$ (Theorem \ref{thm:stab-0}), Corollary \ref{cor:decay}, and
    \begin{align}\label{eqn:delta-1}
    \|\bK^\star\|=\|(\bR+\bB^\top\bP\bB)^{-1}(\bB^\top \bP\bA)\|\leq L_{\bP}L^2/\gamma.
  \end{align}
  By setting $\kappa\geq \okappa$, which ensures $\frac{\Upsilon^2\Psi L(L\Psi + 2L(1+L_{\bP}L^2/\gamma))}{1-\rho^2} \delta^\kappa \leq 1/2$, dividing \eqref{eqn:lyapV} by $V(\bx)$ (assuming $\bx$ is nonzero), and applying $\|\bx\|^2/V( \bx) \geq (1-\rho^2))/\Upsilon^2$, we obtain $\frac{V(\bphi^\kappa(\bx))}{V(\bx)}\leq 1-(1-\rho^2)/2\Upsilon^2$. That is, $V(\cdot)$ is a Lyapunov function. In turn, $\|\bx(t)\|^2\leq \frac{\Upsilon^2}{1-\rho^2} (1-(1-\rho^2)/2\Upsilon^2)^t \|\obx\|^2$, where $\{\bx(t)\}_{t=0}^\infty$ is the state trajectory induced by $\bpi^\kappa(\cdot)$ starting from $\obx$. From this we have $(\Omega,\beta)$-stability of $\bphi^{\kappa}$. Finally, $\beta\in(0,1)$  follows from $\Upsilon>1$ and $\rho\in(0,1)$.

\subsection{Proof of Theorem \ref{thm:regret}}\label{apx:regret}
For simplicity, we let $J^\star(\bx):=J(\bpi^\star(\cdot);\bx)$ and $J^\kappa(\bx):=J(\bpi^\kappa(\cdot);\bx)$. We claim that the following holds:
  \begin{subequations}\label{eqn:del}
    \begin{align}
      \label{eqn:del-1}\Delta_1(\bx)&:=\bpi^\kappa(\bx)^\top \bR\bpi^\kappa(\bx) - \bpi^\star(\bx)^\top \bR\bpi^\star(\bx) \leq L\Psi (2L_{\bP}L^2/\gamma + \Psi)\delta^\kappa\|\bx\|^2\\
      \label{eqn:del-2}\Delta_2(\bx)&:=J^\star(\bphi^\kappa(\bx)) - J^\star(\bphi^\star(\bx))\leq L^2L_{\bP}\Psi (2+2L_{\bP}L^2/\gamma + \Psi) \delta^\kappa\|\bx\|^2.
    \end{align}
  \end{subequations}
  We observe from $\Delta_1(\bx)= (\bpi^\kappa(\bx) + \bpi^\star(\bx))^\top \bR(\bpi^\kappa(\bx) - \bpi^\star(\bx))$, Corollary \ref{cor:decay}, and \eqref{eqn:delta-1} that \eqref{eqn:del-1} holds. Similarly, $\Delta_2(\bx)= (\bphi^\kappa(\bx) + \bphi^\star(\bx))^\top \bP(\bphi^\kappa(\bx) - \bphi^\star(\bx))$, 
  Corollary \ref{cor:decay}, and \eqref{eqn:delta-1} imply \eqref{eqn:del-2}.
  By inspecting the definition of $J^\kappa(\cdot)$ and $J^\star(\cdot)$, we obtain $J^\kappa(\bx^\kappa(t)) - J^\star(\bx^\kappa(t)) = J^\kappa(\bx^\kappa(t+1)) - J^\star(\bx^\kappa(t+1)) + \Delta_1(\bx^\kappa(t)) + \Delta_2(\bx^\kappa(t))$. Summing up this equality over $t=0,\cdots,T-1$ yields
  \begin{align}\label{eqn:main-1}
    J^\kappa(\obx) - J^\star(\obx) = J^\kappa(\bx^\kappa(T)) - J^\star(\bx^\kappa(T)) + \sum_{t=0}^{T-1} \Delta_1(\bx^\kappa(t)) + \sum_{t=0}^{T-1} \Delta_2(\bx^\kappa(t)).
  \end{align}
  From Theorems \ref{thm:stab} and \ref{thm:stab-0}, we have that $J^\kappa(\bx^\kappa(T))\rightarrow 0$ and $J^\star(\bx^\kappa(T))\rightarrow 0$ as $T\rightarrow \infty$. Moreover, from \eqref{eqn:del} and Theorem \ref{thm:stab}, we have $\sum_{t=0}^\infty \Delta_1(\bx^\kappa(t)) \leq \frac{\Omega^2L\Psi (2L_{\bP}L^2/\gamma + \Psi )\delta^\kappa}{1-\beta^2}\|\obx\|^2$ and $\sum_{t=0}^\infty \Delta_2(\bx^\kappa(t)) \leq \frac{\Omega^2L^2L_{\bP}\Psi (2+2L_{\bP}L^2/\gamma + \Psi ) \delta^\kappa}{1-\beta^2}\|\obx\|^2$. \hfill \\
  Therefore, taking $T\rightarrow \infty$ in \eqref{eqn:main-1} yields the desired result.

\subsection{Proof of Theorem \ref{thm:uniform}}\label{apx:uniform}
{We obtain $\|\bQ\|,\|\bR\|,\|\bB\|\leq L_0$ from the block-diagonality (from Assumption \ref{ass:uniform}\ref{ass:uniform-blkdiag}) and boundedness (from Assumption \ref{ass:uniform}\ref{ass:uniform-bdd}). We obtain $\|\bA\|\leq L_0 D$ from Assumptions \ref{ass:uniform}\ref{ass:uniform-bdd}, \ref{ass:degree}, and Lemma \ref{lem:sqrt}. We obtain $\bR\succeq \gamma_0 \bI$ from the block-diagonality (from Assumption \ref{ass:uniform}\ref{ass:uniform-blkdiag}) and Assumption \ref{ass:uniform}\ref{ass:uniform-cvx}.} Now we prove $(L,\alpha)$-stabilizability of $(\bA,\bB)$. Let $\obK_{\cV_k,\cV_k}$ be $(L,\alpha)$-stabilizing feedback for $(\bA_{\cV_k,\cV_k},\bB_{\cV_k,\cV_k})$ (from Assumption \ref{ass:uniform}\ref{ass:uniform-stab}). This and $\obK_{\cV_k,\cV_{k'}}$ from Assumption \ref{ass:uniform}\ref{ass:uniform-stab} allow us to define $\obK$ for the full system. Then, we have $\obPhi_{\cV_k,\cV_k}$ is $(L,\alpha)$-stable, and $\obPhi_{\cV_k,\cV_{k'}}=\bzero$ for $k\neq k'$, where $\obPhi:=\bA-\bB\obK$. One can {observe} that $\obPhi$ is also $(L,\alpha)$-stable, since $\obPhi$ is  block-diagonal, and each block is $(L,\alpha)$-stable. From Assumption \ref{ass:degree} and Lemma \ref{lem:sqrt}, we have $\|\obK\|\leq D L_0$. Therefore, we have constructed $(L,\alpha)$-stabilizing $\obK$; thus, $(\bA,\bB)$ is $(L,\alpha)$-stabilizable.  Theorem \ref{thm:uniform}\ref{thm:uniform-dect} follows from the duality between stabilizability and detectability.

\bibliographystyle{siamplain}
\bibliography{main}

\vspace{0.1cm}
\begin{flushright}
	\scriptsize \framebox{\parbox{2.5in}{Government License: The
			submitted manuscript has been created by UChicago Argonne,
			LLC, Operator of Argonne National Laboratory (``Argonne").
			Argonne, a U.S. Department of Energy Office of Science
			laboratory, is operated under Contract
			No. DE-AC02-06CH11357.  The U.S. Government retains for
			itself, and others acting on its behalf, a paid-up
			nonexclusive, irrevocable worldwide license in said
			article to reproduce, prepare derivative works, distribute
			copies to the public, and perform publicly and display
			publicly, by or on behalf of the Government. The Department of Energy will provide public access to these results of federally sponsored research in accordance with the DOE Public Access Plan. http://energy.gov/downloads/doe-public-access-plan. }}
	\normalsize
\end{flushright}

\end{document}